\documentclass[11pt, a4paper]{article}
\usepackage[latin1]{inputenc}
\usepackage[T1]{fontenc} 
\usepackage{amsmath}
\usepackage{amsfonts}
\usepackage{amssymb}
\usepackage{amsthm}
\usepackage[affil-it]{authblk}
\usepackage{enumitem}
\usepackage[normalem]{ulem}
\usepackage[mathscr]{euscript}
\usepackage{dsfont}
\usepackage{float}
\usepackage{times}
\usepackage{bbm}
\usepackage{xcolor}
\usepackage{xparse}
\usepackage{textcomp}

\usepackage[colorlinks,linkcolor=blue,citecolor=blue,urlcolor=blue]{hyperref}

\usepackage[capitalise]{cleveref}
\usepackage[left=3cm, right=3cm, bottom=2cm, top=2cm]{geometry}

\newtheorem{theorem}{Theorem}[section]
\newtheorem{lemma}[theorem]{Lemma}
\newtheorem{proposition}[theorem]{Proposition}
\newtheorem{corollary}[theorem]{Corollary}

\theoremstyle{definition}
\newtheorem{definition}[theorem]{Definition}

\newtheorem{condition}[theorem]{Condition}

\theoremstyle{remark}
\newtheorem{remark}[theorem]{Remark}
\newtheorem{example}[theorem]{Example}

\numberwithin{equation}{section}

\crefname{example}{Example}{Examples}
\Crefname{example}{Example}{Examples}

\crefname{assumption}{Assumption}{Assumptions}
\Crefname{assumption}{Assumption}{Assumptions}

\crefname{condition}{Condition}{Conditions}
\Crefname{condition}{Condition}{Conditions}

\let\para\S
\crefformat{section}{#2\para#1#3}
\crefformat{subsection}{#2\para#1#3}
\crefformat{subsubsection}{#2\para#1#3}

\setlist{topsep=1ex, itemsep=0.5ex, before={\setlist{topsep=-.5ex}}}

\newcommand{\A}{\ensuremath{\mathscr{A}}}
\renewcommand{\a}{\ensuremath{\mathfrak{a}}}
\newcommand{\f}{\ensuremath{\frac}}
\newcommand{\B}{\ensuremath{\mathcal{B}}}
\newcommand{\C}{\ensuremath{\mathcal{C}}}
\newcommand{\D}{\ensuremath{\mathcal{D}}}

\renewcommand{\S}{{\mathrm{S}}}

\newcommand{\F}{\ensuremath{\mathcal{F}}}
\newcommand{\G}{\ensuremath{\mathcal{G}}}
\renewcommand{\H}{\ensuremath{\mathscr{H}}}
\newcommand{\cH}{\ensuremath{\mathcal{H}}}

\newcommand{\h}{\ensuremath{\mathfrak{h}}}
\newcommand{\K}{\ensuremath{\mathscr{K}}}
\renewcommand{\L}{\ensuremath{\mathcal{L}}}
\newcommand{\N}{\ensuremath{\mathbb{N}}}
\renewcommand{\P}{\ensuremath{\mathcal{P}}}
\newcommand{\R}{\ensuremath{\mathbb{R}}}
\newcommand{\X}{\ensuremath{\mathcal{X}}}
\newcommand{\1}{\ensuremath{\mathds{1}}}
\newcommand{\T}{\ensuremath{\mathcal{T}}}

\newcommand{\W}{\ensuremath{\mathcal{W}}}
\newcommand{\bW}{\ensuremath{\mathbb{W}}}

\newcommand{\sW}{\ensuremath{\mathsf{W}}}

\def\<{\langle}
\def\>{\rangle}

\def\G{{\mathcal G}}
\def\I{\mathcal I}
\def\vsigma{\varsigma}
\def\Leb{\mathop{\mathrm {Leb}}}
\NewDocumentCommand{\Lin}{om}{\IfNoValueTF{#1}{L(\R^{#2},\R^{#2})}{L(\R^{#1},\R^{#2})}}
\NewDocumentCommand{\Cb}{om}{\IfNoValueTF{#1}{\C_b^{#2}}{\C_b^{#2,#1}}}

\def\epsilon{\varepsilon}
\def\le{\leq}
\NewDocumentCommand{\Lip}{om}{\IfNoValueTF{#1}{|#2|_{\mathrm{Lip}}}{|#2|_{\mathrm{Lip};\,#1}}}

\newcommand{\define}{\ensuremath\triangleq}
\newcommand{\expec}[1]{\mathbb{E}[#1]}
\def\E{\mathbb E}
\newcommand{\Expec}[1]{\mathbb{E}\left[#1\right]}

\newcommand{\braket}[1]{\ensuremath\langle#1\rangle}
\newcommand{\Braket}[1]{\ensuremath\left\langle#1\right\rangle}
\newcommand{\prob}{\ensuremath\mathbb{P}}

\newcommand{\TV}[1]{\ensuremath{\left\|#1\right\|_\mathrm{TV}}}

\renewcommand{\geq}{\geqslant}
\renewcommand{\leq}{\leqslant}

\def\${|\!|\!|}
\def\B{{\mathcal B}}
\def\F{{\mathcal F}}
\def\ge{\geq}
\def\1{\mathbf 1}
\newcommand{\vertiii}[1]{{\left\vert\kern-0.25ex\left\vert\kern-0.25ex\left\vert #1 \right\vert\kern-0.25ex\right\vert\kern-0.25ex\right\vert}}
\newcommand{\rom}[1]{(\textup{\uppercase\expandafter{\romannumeral#1}})}

\makeatletter
\newcommand{\substackal}[1]{%
  \vcenter{%
    \Let@ \restore@math@cr \default@tag
    \baselineskip\fontdimen10 \scriptfont\tw@
    \advance\baselineskip\fontdimen12 \scriptfont\tw@
    \lineskip\thr@@\fontdimen8 \scriptfont\thr@@
    \lineskiplimit\lineskip
    \ialign{\hfil$\m@th\scriptstyle##$&$\m@th\scriptstyle{}##$\hfil\crcr
      #1\crcr
    }%
  }%
}
\makeatother

\definecolor{LB}{rgb}{0.29, 0.63, 0.73}

\begin{document}
\title{Slow-Fast Systems with Fractional Environment and Dynamics}
\author{Xue-Mei Li and Julian Sieber\thanks{Email addresses: \{\href{mailto:xue-mei.li@imperial.ac.uk}{xue-mei.li}, \href{mailto:j.sieber19@imperial.ac.uk}{j.sieber19}\}@imperial.ac.uk}}
\affil{Department of Mathematics, Imperial College London, UK}
\date{\today}
\maketitle

\begin{abstract}
  \noindent We prove a fractional averaging principle for interacting slow-fast systems. The mode of convergence is in H\"older norm in probability. The main technical result is a quenched ergodic theorem on the conditioned fractional dynamics. We also establish geometric ergodicity for a class of fractional-driven stochastic differential equations, improving a recent result of Panloup and Richard.
  \vspace{0.2cm}
\end{abstract}

{\hypersetup{hidelinks}
\tableofcontents
}

\section{Introduction and Main Results}

We study slow-fast systems driven by fractional Brownian motions (fBm):
\begin{alignat}{4}
  dX_t^\varepsilon&=f(X_t^\varepsilon,Y_t^\varepsilon)\,dt+g(X_t^\varepsilon, Y_t^\varepsilon)\,dB_t, &\qquad X_0^\varepsilon&=X_0, \label{eq:slow}\\
  dY_t^\varepsilon&=\frac{1}{\varepsilon}b(X_t^\varepsilon,Y_t^\varepsilon)\,dt+\frac{1}{\varepsilon^{\hat{H}}}\sigma\,d\hat{B}_t, &\qquad Y_0^\varepsilon&=Y_0, \label{eq:fast}
\end{alignat}
where $B$ and $\hat{B}$ are independent fBms on an underlying complete probability space $(\Omega, \F,\prob)$ with Hurst parameters $H\in(\frac12,1)$ and $\hat{H}\in(1-H,1)$, respectively. Here, $g:\R^d\times\R^n\to\Lin[m]{d}$ and $\sigma\in\Lin{n}$ is non-degenerate. As the scale parameter $\varepsilon>0$ is taken to $0$, one hopes that the \emph{slow motion} $X^\varepsilon$ is well approximated by an \emph{effective dynamics} $\bar{X}$. For $H=\hat{H}=\frac12$, this convergence has been studied by myriad authors since the seminal works of Bogolyubov-Mitropol{\textquotesingle}ski\u{\i} \cite{Bogolyubov1955} and Hasminskii \cite{Hasminskii1968}, see e.g. the monographs and survey articles \cite{Freidlin2012,Skorokhod2002,Pavliotis2008,Berglund2006,Liu2012,Li2018} and references therein for a comprehensive overview. It is still a very active research area \cite{Liu2020,Roeckner2020,Roeckner2020a}.

For $H,\hat{H}\neq\frac12$, the SDEs \eqref{eq:slow}--\eqref{eq:fast} provide a suitable model for economic, medical, and climate phenomena exhibiting a genuinely non-Markovian behavior in both the system and its environment. It is for example very well known that neglecting temporal memory effects in climate modeling by resorting to a diffusion model results in prediction notoriously mismatching observational data \cite{Ashkenazy2003,Karner2002,Davidsen2010,Barboza2014}. It thus became widely popular to use fBm in climate modeling \cite{Sonechkin1998,Yuan2014,Eichinger2020}.

While slow-fast systems with fractional noise have seen a tremendous spike of interest in the last two years \cite{Bourguin-ailus-Spiliopoulos-typical,Bourguin-Gailus-Spiliopoulos,Hairer2020,Pei-Inaham-Xu, Pei-Inaham-Xu2,Han2021}, all of these works resort to Markovian, strongly mixing fast processes by choosing $\hat{H}=\frac12$ in \eqref{eq:fast}. The main contribution of this article is to establish the convergence $X^\varepsilon\to\bar{X}$ even for a \emph{non-Markovian} fast dynamics by allowing $\hat{H}\neq\frac12$. It hardly comes as a surprise that this renders the analysis much more delicate and it is not clear at all if an averaging principle can even hold for a fractional, \emph{non-mixing} environment. In fact, the usual assumption in the aforementioned works on Markovian averaging principles is a strong mixing condition with an algebraic rate \cite{Heunis1994,Abourashchi2010}. This condition is essentially never satisfied for a fractional dynamics \cite{Bai2016}.

Recent work of Hairer and the first author of this article suggests the following ansatz for the effective dynamics:
\begin{equation}\label{eq:effective_dynamics}
  d\bar{X}_t=\bar{f}(\bar{X}_t)\,dt+\bar{g}(\bar{X}_t)\,dB_t,\qquad \bar{X}_0=X_0,
\end{equation}
where $\bar{f}(x)\define\int f(x,y)\,\pi^x(dy)$ and similar for $\bar{g}$ \cite{Hairer2020}. For $\hat{H}=\frac12$, this work showed that the average is taken with respect to the unique invariant $\pi^x$ of the fast dynamics with \emph{frozen} slow input
\begin{equation}\label{eq:frozen_fast}
  dY_t^x=b(x,Y_t^x)\,dt+\sigma\,d\hat{B}_t.
\end{equation} 
For $\hat{H}\neq\frac12$, it is \emph{a priori} not clear what $\pi^x$ should be. We show that it is the one-time marginal of the unique stationary path space law $\prob_{\pi^x}\in\P\big(\C(\R_+,\R^n)\big)$, see \cref{sec:physical_solution} for details. [Here and in the sequel, $\P(\X)$ denotes the set of Borel probability measures on a Polish space $\X$.]

In addition to standard regularity requirements ensuring well-posedness of the slow-fast system (see \cref{cond:feedback} below), we shall impose a contractivity condition on the drift in \eqref{eq:fast}:
\begin{definition}\label{define-semi-contractive}
  Let $\lambda, R\geq 0$ and $\kappa>0$. We write $\S(\kappa, R, \lambda)$ for the set of Lipschitz continuous functions $b:\R^n\to\R^n$ satisfying
  \begin{equation}\label{eq:semicontractive}
  \Braket{b(x)-b(y),x-y}\leq\begin{cases}
  -\kappa|x-y|^2, & |x|,|y|\geq R,\\
  \lambda|x-y|^2, &\text{otherwise}.\\
  \end{cases}
\end{equation}
\end{definition}

Note that $\lambda$ may be smaller than $\Lip{b}$, whence its prescription is not necessarily redundant. If $b=-\nabla V$ is a gradient vector field with potential $V$, then \eqref{eq:semicontractive} is equivalent to $V$ being at most $\lambda$-concave on $|x|<R$ and $\kappa$-convex on $|x|\geq R$. If $V\in\C^2(\R^n)$, these requirements are in turn equivalent to $\nabla^2V\geq-\lambda$ and $\nabla^2 V\geq\kappa$ on the respective sets.

\begin{theorem}[Fractional Averaging Principle]\label{thm:feedback_fractional}
  Consider the slow-fast system \eqref{eq:slow}--\eqref{eq:fast}. Suppose that $f,g\in\Cb{2}$ and $b$ satisfies \cref{cond:feedback}. Let $\alpha<H$ and $\kappa,R>0$. Then there is a number $\lambda_0>0$ such that, if $b(x,\cdot)\in\S\big(\kappa, R,\lambda_0\big)$ for every $x\in\R^d$, all of the following hold:
  \begin{itemize}
    \item For every $x\in\R^d$, there exists a unique stationary path space law $\prob_{\pi^x}\in\P\big(\C(\R_+,\R^n)\big)$ for the frozen fast dynamics \eqref{eq:frozen_fast}.
    \item Let $\pi^x\in\P(\R^n)$ be the one-time marginal of $\prob_{\pi^x}$. If 
    \begin{equation*}
      x\mapsto\bar{g}(x)\define\int_{\R^n}g(x,y)\,\pi^x(dy)\in\C_b^2\big(\R^d,\Lin[m]{d}\big),
    \end{equation*}
    then there is a unique pathwise solution to \eqref{eq:effective_dynamics} and $X^\varepsilon\to\bar{X}$ as $\varepsilon\to 0$ in $\C^\alpha\big([0,T],\R^d\big)$ in probability for any $T>0$.
  \end{itemize} 
\end{theorem}

The regularity of $\bar{g}$ not only hinges on the regularity of $g$ but also on the fast dynamics. First we note that the requirement on $\bar{g}$ clearly holds for a diffusion coefficient depending only on the slow motion $X^\varepsilon$: 
\begin{equation*}
  dX_t^\varepsilon=f(X_t^\varepsilon,Y_t^\varepsilon)\,dt+g(X_t^\varepsilon)\,dB_t.
\end{equation*}
Another class of examples is provided by \cref{cor:smooth} below. 

The technical core of the proof of \cref{thm:feedback_fractional} is a quantitative \emph{quenched ergodic theorem} on the conditional evolution of the process \eqref{eq:frozen_fast}. We prove this by means of a control argument, which is of independent interest. In fact, it allows us to improve recent work of Panloup and Richard \cite{Panloup2020} by establishing geometric ergodicity for a class of SDEs driven by additive fractional noise. To our best knowledge, this is the first result achieving an exponential convergence rate for a fractional dynamics (excluding the trivial instance of an everywhere contractive drift).

Let $\TV{\mu}\define\sup_A|\mu(A)|$ denote the total variation norm, $\W^p$ be the $p$-Wasserstein distance, and $\bW^p$ be the Wasserstein-like metric for generalized initial conditions introduced in \cref{def:wasserstein}.
 
\begin{theorem}[Geometric Ergodic Theorem]\label{thm:geometric}
Let $(Y_t)_{t\geq 0}$ be the solution to the SDE
\begin{equation}\label{eq:sde_intro}
  dY_t=b(Y_t)\,dt+\sigma\,dB_t
\end{equation}
started in the generalized initial condition $\mu$, where $\sigma\in\Lin{n}$ is non-degenerate and $B$ is an fBm with Hurst parameter $H\in(0,1)$. Then, for any $p\geq 1$ and any $\kappa,R>0$, there exists a $\Lambda=\Lambda(\kappa,R,p)>0$ such that, whenever $b\in\S\big(\kappa,R,\Lambda\big)$, there is a unique invariant measure $\I_\pi$ for \eqref{eq:sde_intro} in the sense of \cref{initial-condition}. Moreover, 
    \begin{equation}\label{eq:wasserstein_time_t}
        \W^p(\mathcal{L}(Y_t),\pi)\leq Ce^{-ct} \bW^p\big(\mu,\I_{\pi}\big) \qquad \forall\, t\geq 0
    \end{equation}
    and
    \begin{equation}\label{eq:tv_process}
        \TV{\L(Y_{\cdot+t})-\prob_\pi}\leq Ce^{-ct}\bW^1\big(\mu,\I_\pi\big) \qquad \forall\, t\geq 0,
    \end{equation}
    where $c,C>0$ are numerical constants independent of $t\geq 0$ and $\mu$.
\end{theorem}

The work \cite{Hairer2005} already contained a result on the rate of convergence. There, the author assumed an off-diagonal contraction condition, see \cref{cond:off_diagonal} below, and obtained an algebraic rate in \eqref{eq:tv_process}. Very recently Panloup and Richard \cite{Panloup2020} studied $b\in\S(\kappa,R,0)$ for which they found a rate of order $e^{-Dt^\gamma}$ for some $\gamma<\frac23$ in both \eqref{eq:wasserstein_time_t} and \eqref{eq:tv_process}. Albeit these works did not require a global Lipschitz condition on the drift for Hurst parameters $H<\frac12$, we emphasize that they do impose this assumption for $H>\frac12$ to obtain \eqref{eq:tv_process}. This is due to the lack of regularity of a certain fractional integral operator. \Cref{thm:geometric} thus provides a genuine ramification of the results of \cite{Panloup2020} in the latter case. We note that similarly to the work of Panloup and Richard, the Wasserstein decay \eqref{eq:wasserstein_time_t} also holds for more general Gaussian driving noises with stationary increments. We shall briefly comment on this in \cref{sec:geometric_ergodicity}.

With the spiking interest in numerical methods based on the generalized Langevin equation with memory kernel \cite{Chak2020,Leimkuhler2020}, \cref{thm:geometric} and the quenched quantitative ergodic theorem underpinning it can give a better theoretical understanding. A first step would be to derive quantitative estimates on the constants $c$, $C$, and $\Lambda$; a possible pathway is outlined in \cref{rem:constant_xi} below. It is an interesting open question if there is indeed a finite threshold value of $\Lambda$ beyond which the exponential rates \eqref{eq:wasserstein_time_t}--\eqref{eq:tv_process} no longer hold. As established by Eberle, such a transition from exponential to sub-exponential rates does not happen in case $H=\frac12$ \cite{Eberle2016}.

\begin{example}
  Let us give an example of a drift not covered by the sub-exponential convergence theorems of \cite{Panloup2020}. Consider the double-well potential
  \begin{equation*}
      V(x)=\alpha|x|^4-\beta|x|^2
  \end{equation*}
  for $\alpha,\beta>0$. We modify $V$ outside of a compact such that its Hessian is bounded. Set $b=-\nabla V$. It is clear that $b\notin\bigcup_{\kappa,R>0}\S(\kappa,R,0)$ as soon as $\beta>0$. However, for $\frac{\beta}{\alpha}$ sufficiently small, \cref{thm:geometric} furnishes an exponential rate of convergence.
\end{example}

\paragraph{Outline of the article.} The next section features a brief overview of preliminary material. In \cref{sec:convergence}, we prove the quantitative quenched ergodic theorem and deduce \cref{thm:geometric}. The proof of \cref{thm:feedback_fractional} is concluded in \cref{sec:feedback}.
\paragraph{Acknowledgements.} We would like to thank the anonymous referees for their careful reading and helpful comments. Partial support from the EPSRC under grant no. EP/S023925/1 is also acknowledged.

\section{Preliminaries}\label{sec:preliminaries}

Recall that one-dimensional fractional Brownian motion with Hurst parameter $H\in(0,1)$ is the centered Gaussian process $(B_t)_{t\geq 0}$ with
\begin{equation*}
  \Expec{(B_t-B_s)^2}=|t-s|^{2H},\qquad s,t\geq 0.
\end{equation*}
To construct $d$-dimensional fBm one lets the coordinates evolve as independent one-dimensional fBms with the same Hurst parameter. We will make frequent use of the following classical representation of one-dimensional fBm as a fractional integral of a two-sided Wiener process $(W_t)_{t\in\R}$, which is due to Mandelbrot and van Ness \cite{Mandelbrot1968}:
\begin{equation}\label{eq:mandelbrot}
 B_t=\alpha_H\int_{-\infty}^0 (t-u)^{H-\frac12}-(-u)^{H-\frac12}\,dW_u+\alpha_H\int_0^t(t-u)^{H-\frac12}\,dW_u,\qquad t\geq 0.
\end{equation}
Here, $\alpha_H>0$ is some explicitly known normalization constant and we also write $B_t=\bar B_t+\tilde B_t$.

\subsection{Invariant Measures of Fractional SDEs} \label{sec:physical_solution}

Albeit being certainly non-Markovian on its own, the solution to \eqref{eq:sde_intro} can actually be cast as the marginal of an infinite-dimensional Feller process $Z_t\define\big(Y_t,(W_s)_{s\leq t}\big)$ with values in $\R^n\times\H_H$. Here, $W$ is the two-sided Wiener process driving the equation through \eqref{eq:mandelbrot} and $\H_H$ is a H\"older-type space of paths $\R_-\to\R^n$ supporting the Wiener measure $\sW$. More concretely, $\H_H$ is the closure of the space $\{f\in\C_c^\infty(\R_-,\R^n):\,f(0)=0\}$ in the norm
\begin{equation*}
  \|f\|_{\H_H}\define\sup_{s,t\leq 0}\frac{\big|f(t)-f(s)\big|}{|t-s|^{\frac{1-H}{2}}\sqrt{1+|t|+|s|}}.
\end{equation*} 
To ensure that this construction actually furnishes a solution to \eqref{eq:sde_intro}, we of course have to assume that the law of the second marginal of $Z$ coincides with $\sW$ for each time $t\geq 0$. This motivates the following definition:
\begin{definition}[\cite{Hairer2005}]\label{initial-condition}
A measure $\mu\in\P(\R^n\times\H_H)$ with $\Pi_{\H_H}^*\mu=\sW$ is called a \emph{generalized initial condition}. A generalized initial condition $\I_\pi$, which is invariant for the Feller process $Z$ is called an \emph{invariant measure} for the SDE \eqref{eq:sde_intro}. We write $\pi\define\Pi_{\R^n}^*\I_\pi$ for the first marginal and $\prob_\pi\in\P\big(\C(\R_+,\R^n)\big)$ for the law of the first coordinate of $Z$ when started in $\I_\pi$.
\end{definition} 

By only adding the past of the driving noise to the auxiliary process $Z$, Hairer's framework rules out the existence of `unphysical' invariant measures, which frequently occur in the theory of random dynamical systems, see \cite{Hairer2009} for details.

There are only a few examples for which the invariant measure can be written down explicitly:
\begin{example}\label{ex:disintegration_fou}
  Let $Y$ be the fractional Ornstein-Uhlenbeck process \cite{Cheridito2003}, that is,
  \begin{equation*}
    dY_t=-Y_t\,dt+dB_t.
  \end{equation*}
  Then it is well known that its invariant measure is given by
  \begin{equation*}
    \I_\pi(dy,dw)=\delta_{F(w)}(dy)\sW(dw),\qquad F(w)\define-\int_{-\infty}^0 e^s D_Hw(s)\,ds,
  \end{equation*}
  where $D_H:\H_H\to\H_{1-H}$ is a continuous linear operator switching between Wiener and fBm paths, see \cite[Eq. (3.6)]{Hairer2005} for the precise definition. The first marginal of $\I_\pi$ and the stationary path space law are given by
  \begin{equation*}
    \pi=\L\left(\int_{-\infty}^0 e^s\,dB_s\right)\quad\text{and}\quad\prob_\pi=\L\left(\int_{-\infty}^t e^s\,dB_s\right)_{t\geq 0}.
  \end{equation*}
  
\end{example}

\begin{remark}
  The invariant measure of \eqref{eq:sde_intro} is in general not of product form.
\end{remark}

Since $\sigma\in\Lin{n}$ is non-degenerate, one can show that there is an isomorphism between the strictly stationary solutions to \eqref{eq:sde_intro} and the set of invariant measures (provided one quotients the latter by the equivalence relation identifying generalized initial initial conditions which generate the same evolution in the first marginal). It is also not hard to prove the following:
\begin{proposition}[\cite{Hairer2005}]\label{prop:existence_invariant_measure}
  If $\sigma\in\Lin{n}$ and $b\in\S(\kappa,R,\lambda)$ for some $\kappa>0$, $R,\lambda\geq 0$, then there exists an invariant measure for \eqref{eq:sde_intro} in the sense of \cref{initial-condition}. Moreover, $\I_\pi$ has moments of all orders.
\end{proposition}

The conclusion of \cref{prop:existence_invariant_measure} actually holds for a merely locally Lipschitz off-diagonal large scale contractive drift (see \cref{cond:off_diagonal} below). See also \cite{Hairer2007,Deya2019} for versions for multiplicative noise. Finally, we introduce a Wasserstein-type distance for generalized initial conditions:
\begin{definition}\label{def:wasserstein}
Let $\mu$ and $\nu$ be generalized initial conditions. Let $\mathscr{C}_{\Delta}(\mu,\nu)$ denote the set of couplings of $\mu$ and $\nu$ concentrated on the diagonal $\Delta_{\H_H}\define\{(w,w^\prime)\in\H_H^2:\,w=w^\prime\}$. For $p\geq 1$, we set
  \begin{equation*}
  \bW^p(\mu,\nu)\define\inf_{\rho\in\mathscr{C}_\Delta(\mu,\nu)}\left(\int_{(\R^n\times\H_H)^2}|x-y|^p\,\rho(dx,dw,dy,dw^\prime)\right)^{\frac1p}.
\end{equation*}
\end{definition}

Note that clearly $\W^p\big(\Pi_{\R^n}^*\mu,\Pi_{\R^n}^*\nu\big)\leq \bW^p(\mu,\nu)$ and the inequality is strict in general.

\subsection{Large Scale Contractions}

Known ergodic theorems on \eqref{eq:sde_intro} require either a Lyapunov-type stability or a large scale contractivity condition on the drift $b$. The former indicates that once far out, the solutions have the tendency to come back to a neighborhood of the origin. Under this condition, it is conceivable that two distinct solutions can come back from diverging routes, thus allowing to couple them. The Lyapunov stability condition was used in \cite{Fontbona2017,Deya2019} for multiplicative noise. 

A large scale contraction on the other hand will force two solutions to come closer once they have left a ball $B_R$ of sufficiently large radius $R>0$. The following two conditions appeared in previous works:

\begin{condition}[Off-diagonal large scale contraction, \cite{Hairer2005}]\label{cond:off_diagonal}
There exist numbers $\tilde \kappa>0$ and $D,\lambda\geq 0$ such that  
\begin{equation}\label{quasi-contr}
  \Braket{b(x)-b(y),x-y}\leq \big(D-\tilde \kappa|x-y|^2\big)\wedge\big(\lambda|x-y|^2\big)\qquad \forall\,   x,y\in\R^n.
\end{equation}
\end{condition}

\begin{condition}[Large scale contraction, \cite{Panloup2020}]
There exist numbers $R\ge 0$ and $\kappa>0$ such that 
  \begin{equation}\label{contractive}
    \Braket{b(x)-b(y),x-y}\leq    -\kappa|x-y|^2   \qquad \forall\, x,y\in \R^n\setminus B_R.  
  \end{equation}
\end{condition}

\begin{example}
  The function $b(x)=x-x^3$ is a large scale contraction. 
\end{example}

We will later use the following standard result, a slightly weaker version of which was proven in \cite[Lemma 5.1]{Panloup2020}.
\begin{lemma}\label{lem:bigger_ball}
If $b$ is locally Lipschitz continuous and satisfies the large scale contraction condition \eqref{contractive}, then for any $\bar{\kappa}\in(0,\kappa)$, there is an $\bar{R}>0$ such that
    \begin{equation*}
        \braket{b(x)-b(y),x-y}\leq -\bar{\kappa}|x-y|^2  \qquad \forall\, y\in\R^n,\, |x|>\bar{R}.
    \end{equation*}
\end{lemma}
\begin{proof}
 Since $\braket{b(x)-b(y),x-y}\leq  -{\kappa}|x-y|^2$ for $x$ and $y$ outside of the ball $B_R$, 
we only need to show that the required contraction holds for any $|y|\le R$ and $|x|>\bar R$. 
Fix such $x$ and $y$. 

Without loss of generality, we may also assume that $\bar{R}\geq R+1$. Then there is a $\beta\in(0,1)$ such that $z_\beta\define (1-\beta)x+\beta y$ has norm $|z_\beta|=R+1$. 
Since $x-y=\f 1 \beta(x-z_\beta)$ and, since $x, z_\beta$ are outside of $B_R$, 
    \begin{equation*}
        \braket{b(x)-b(z_\beta),x-y}\leq-\f 1 \beta \kappa|x-z_\beta|^2=-\kappa\beta|x-y|^2.
    \end{equation*}
Let $K\define\Lip[B_{R+1}]{b}$ denote the Lipschitz constant of $b$ on $B_{ R+1}$. Since $|z_\beta-y|=(1-\beta)|x-y|$, it holds that
    \begin{align*}
        \braket{b(x)-b(y),x-y}&=\braket{b(x)-b(z_\beta),x-y} +\<b(z_\beta)-b(y), x-y\>\\
        & \leq -\kappa\beta|x-y|^2+K(1-\beta)|x-y|^2
    \end{align*}
Since $\beta$ is the length of the proportion of the line segment outside of $B_{R+1}$,
we can choose it as close to $1$ as we like by choosing $\bar R$ sufficiently large $\big(\beta=\f{|x-z_\beta|}{|x-y|}\geq\frac{|x|-R-1}{|x|+R}\geq\frac{\bar{R}-R-1}{\bar{R}+R}\big)$.
\end{proof}

\begin{remark}\label{rem:large_scall_off_diagonal}
\leavevmode
\begin{enumerate}
  \item Let $b: \R^n\to \R^n$ be a globally Lipschitz continuous function. Then the large scale contraction condition \eqref{contractive} is equivalent to $b\in\bigcup_{\lambda>0}\S(\kappa,R,\lambda)$. In view of \cref{lem:bigger_ball}, condition \eqref{eq:semicontractive} also holds for a merely locally Lipschitz continuous $b$ at the cost of a smaller contractive rate and a bigger contractive range. In fact, choose $\bar\kappa\in(0,\kappa)$ and let $\bar R>R$ be the corresponding radius furnished by \cref{lem:bigger_ball}. This gives \eqref{eq:semicontractive} with $\kappa\rightsquigarrow\bar{\kappa}$, $R\rightsquigarrow\bar{R}$, and $\lambda\rightsquigarrow \Lip[B_{\bar{R}}]{b}$.
  \item\label{it:off_diagonal} The off-diagonal large scale contraction condition is weaker than the large scale contraction condition.  With the former, there may be no $\kappa>0$ such that \eqref{contractive} holds in the region $\{|x-y| \le \f D {2\tilde \kappa}\} \cap \{|x|\ge R, |y|\ge R\}$. On the other hand, if \eqref{contractive} holds and $b$ is locally Lipschitz continuous, we can choose any $\tilde \kappa<\kappa$. In fact, denoting the radius from \cref{lem:bigger_ball} by $\bar R>0$, one only needs to show \eqref{quasi-contr} when both $x$ and $y$ are in $B_{\bar R}$. To this end, we pick $\lambda=\Lip[B_{\bar{R}}]{b}$ and $D\geq\sup_{x,y\in B_{\bar{R}}}(\tilde\kappa+\lambda)|x-y|^2$.
\end{enumerate}
\end{remark}

\section{The Conditional Evolution of Fractional Dynamics}\label{sec:convergence}

To derive strong $L^p$-bounds on the H\"older norm of the slow motion in \cref{sec:feedback} below, we need to study the conditional distribution of the evolution \eqref{eq:frozen_fast}. Unlike the Markovian case, the conditioning changes the dynamics and the resulting evolution may \emph{no longer} solve the original equation. We will show that, in the limit $t\to\infty$, the law of the conditioned dynamics still converges to $\pi^x$, the first marginal of the invariant measure for the fast dynamics with frozen slow input \eqref{eq:frozen_fast}. The rate of convergence is however slower (only algebraic rather than exponential).

Let us first state the regularity assumption imposed in \cref{thm:feedback_fractional}. For this we introduce a convenient notation, which we shall frequently use in the sequel. We write $a\lesssim b$ if there is a constant $C>0$ such that $a\leq C b$. The constant $C$ is independent of any ambient parameters on which $a$ and $b$ may depend.

\begin{condition}\label{cond:feedback}
  The drift $b:\R^d\times\R^n\to\R^n$ satisfies the following conditions:
  \begin{itemize}
    \item \emph{Linear growth:}
    \begin{equation*}
      |b(x,y)|\lesssim 1+|x|+|y|, \qquad \forall\, x\in\R^d,y\in\R^n.
    \end{equation*}
    \item \emph{Uniformly locally Lipschitz in the first argument:} For each $R>0$, there is an $L_R>0$ such that
    \begin{equation*}
      \sup_{y\in\R^n}|b(x_1,y)-b(x_2,y)|\leq L_R|x_1-x_2|, \qquad \forall\, |x_1|,|x_2|\leq R.
    \end{equation*}
    \item \emph{Uniformly Lipschitz in the second argument:} There is an $L>0$ such that
    \begin{equation*}
      \sup_{x\in\R^d}|b(x,y_1)-b(x,y_2)|\leq L|y_1-y_2|, \qquad \forall\, y_1,y_2\in\R^n.
    \end{equation*}
  \end{itemize}
\end{condition}

Let $(\F_t)_{t\geq 0}$ be a complete filtration to which $\hat{B}$ is adapted. For any continuous, $(\F_t)_{t\geq 0}$-adapted, $\R^d$-valued process $X$ with continuous sample paths, and any $\varepsilon>0$, the equation
\begin{equation}\label{eq:general_flow}
  d\Phi_{t}^X=\frac{1}{\varepsilon}b\big(X_t,\Phi_{t}^X\big)\,dt+\frac{1}{\varepsilon^{\hat{H}}}\sigma\,d\hat{B}_t,\qquad \Phi_{t}^{X}=y,
\end{equation}
has a unique global pathwise solution under \cref{cond:feedback}, see \cref{lem:comparison} below. The flow $\Phi_{s,t}^X(y)$ associated with \eqref{eq:general_flow} is therefore well defined. An important special case of \eqref{eq:general_flow} is when the extrinsic process is given by a fixed point $x\in\R^d$. For this we reserve the notation $\bar{\Phi}^x$:
\begin{equation}\label{eq:general_flow-fixed-x}
  d\bar{\Phi}_t^x=\frac{1}{\varepsilon}b(x,\bar{\Phi}_t^x)\,dt+\frac{1}{\varepsilon^{\hat{H}}}\sigma\,d\hat{B}_t,\qquad \bar{\Phi}_0^x=y.
\end{equation}
We would like the reader to observe that the dependency of flows on the scale parameter $\varepsilon>0$ is suppressed in our notation. Note that, by self-similarity, sending $\varepsilon\to 0$ in \eqref{eq:general_flow-fixed-x} is equivalent to keeping $\varepsilon=1$ fixed and taking $t\to\infty$. As the $\varepsilon$-dependence of the flows \eqref{eq:general_flow}--\eqref{eq:general_flow-fixed-x} will play a key r\^ole in \cref{sec:feedback}, we choose to introduce a new notation in case $\varepsilon=1$, which is used throughout the rest of this section: 
\begin{definition}\label{def:flow}
  Let $\h\in\C_0(\R_+,\R^n)\define\big\{f\in\C(\R_+,\R^n):\,f(0)=0\big\}$ and $x\in\R^d$. We denote the flow of the ordinary differential equation
  \begin{equation}\label{eq:ode_solution}
    dy_t=b(x,y_t)\,dt+d\h_t
  \end{equation}
  by $\Psi^x_{s,t}(y,\h)$, where $y\in\R^n$ and $0\leq s\leq t$. It is given by the solution to the integral equation
  \begin{equation*}
    \Psi_{s,t}^x(y,\h)=y+\int_s^t b\big(x,\Psi_{s,r}^x(y,\h)\big)\,dr+\h_t-\h_s.
  \end{equation*}
   We also use the abbreviation $\Psi^x_{t}\define \Psi^x_{0,t}$.
\end{definition}

Under \cref{cond:feedback}, \eqref{eq:ode_solution} is well posed and it follows that $\Psi_{s,t}^x(y, \h)=\Psi_{t-s}^x(y,\theta_s \h)$ for each $0\leq s\leq t$ and $y\in\R^n$, where $\theta_sf=f(\cdot+s)-f(\cdot)$ is the Wiener shift operator on the path space. If $x\in\R^d$, $y\in\R^n$, or $\h\in\C_0(\R_+,\R^n)$ are random, we understand \cref{def:flow} pathwise for each fixed sample $\omega\in\Omega$. The solutions to \eqref{eq:general_flow} and \eqref{eq:general_flow-fixed-x} are also understood in this sense.

\subsection{Processes with a Locally Independent Increment Decomposition}\label{sec:increment}

The derivation of the conditioned evolution relies on the following simple fact: For $t,h\geq 0$, we have
\begin{equation}\label{eq:increment_decomposition}
  (\theta_t\hat{B})_h=\hat{B}_{t+h}-\hat{B}_t=\bar{\hat{B}}_h^t+\tilde{\hat{B}}_h^t,
\end{equation}
where, in a slight abuse of notation (the integrand has to be multiplied by the identity matrix),
\begin{equation*}
  \bar{\hat{B}}_h^t\define \alpha_{\hat H}\int_{-\infty}^t\left((t+h-u)^{\hat{H}-\frac12}-(t-u)^{\hat{H}-\frac12}\right)\,d\hat{W}_u,\quad \tilde{\hat{B}}_h^t\define\alpha_{\hat H}\int_t^{t+h} (t+h-u)^{\hat{H}-\frac12}\,d\hat{W}_u.
\end{equation*}
This decomposition is easily obtained by rearranging \eqref{eq:mandelbrot}. For any $t\geq 0$, the two components $\bar{\hat{B}}^t$ and $ \tilde{\hat{B}}^t$ are independent. We call $\bar{\hat{B}}^t$ the \emph{smooth} part of the increment, whereas $\tilde{\hat{B}}^t$ is referred to as the \emph{rough} part. This terminology is based on the fact that, away from the origin, the process $\bar{\hat{B}}^t$ has continuously differentiable sample paths and therefore the `roughness' of $\hat{B}$ essentially comes from $\tilde{\hat{B}}^t$. Indeed, it is not hard to check that $\tilde{\hat{B}}^t$ is of precisely the same H\"older regularity as $\hat{B}$. We also observe that $\tilde{\hat{B}}^t\overset{d}{=}\tilde{\hat{B}}^0\define\tilde{\hat{B}}$ for all $t>0$. 

The process $\tilde{\hat{B}}$ is---up to a prefactor---known as Riemann-Liouville process (or type-II fractional Brownian motion) and was initially studied by L\'evy \cite{Levy1953}. Its use in modelling was famously discouraged in \cite{Mandelbrot1968} due to its overemphasis of the origin and the `regularized' process \eqref{eq:mandelbrot} was proposed instead. In fact as we shall see below, the lack of stationarity of the increments of $\tilde{\hat{B}}$ complicates the analysis of the conditioned evolution. 

\begin{definition}\label{def:ind_increment}
Let  $(\F_t)_{t\geq 0}$ be a complete filtration. An $(\F_t)_{t\geq 0}$-adapted stochastic process $Z$ is said to have a \emph{locally independent decomposition of its increments} with respect to $(\F_t)_{t\geq 0}$ if for any $t\geq 0$, there exists an increment decomposition of the form
$$(\theta_t Z)_h=\tilde Z^t_h+\bar Z^t_h, \qquad h\geq 0,$$
where $\bar Z^t \in \F_t$ and $\tilde Z^t$ is independent of $\F_t$. 
\end{definition}

As seen in \eqref{eq:increment_decomposition}, an fBm $\hat{B}$ has a locally independent decomposition of its increments with respect to any filtration $(\F_t)_{t\geq 0}$ \emph{compatible} with $\hat{B}$. By this we mean that $(\hat{W}_s)_{s\leq t}\in\F_t$ and $(\theta_t\hat{W}_s)_{s\geq t}$ is independent of $\F_t$ for any $t\geq 0$.

\begin{example}\label{example-1}
Let us give some further examples, which will become important later on:
\begin{enumerate}
\item\label{it:rough_decomposition} Let $(\hat W_t)_{t\geq 0}$ be a Wiener process and 
$ \tilde{\hat B}_t\define\alpha_{\hat H}\int_0^{t} (t-u)^{H-\frac12}\,d\hat W_u$ be the Riemann-Liouville process. 
Then, for any $t\ge 0$ and $h\ge 0$, 
 \begin{align}
    (\theta_t\tilde{\hat B})_h&=\alpha_{\hat H}\int_0^t \Big( (t+h-u)^{\hat H-\frac12}-(t-u)^{\hat H-\frac12}\Big)\,d\hat W_u+\alpha_{\hat H}\int_t^{t+h}(t+h-u)^{\hat H-\frac12}\,d\hat W_u\nonumber\\
    &\define Q^t_h+\tilde{\hat B}^t_h.\label{eq:z_t}
  \end{align}
Thus, $\tilde{\hat B}$ admits a locally independent decomposition of its increments with respect to any filtration compatible with $\hat{B}$.

\item Another example, given in \cite{Gehringer-Li-2020, Gehringer-Li-2020-1}, is the stationary fractional Ornstein-Uhlenbeck process $Z_t=\int_{-\infty }^t e^{-(t-s)}\,d\hat{B}_s$. More generally, it is clear that $Z_t=\int_{-\infty }^t \mathfrak{G}(s,t)\,d\hat{B}_s$ with a suitable kernel $\mathfrak{G}$ also has this property.

\item\label{it:smooth_decomposition} Albeit not being a direct instance of \cref{def:ind_increment}, it is also interesting to observe a \emph{fractal} property of $\hat{B}$: The smooth part of the increment has an independent decomposition as $\bar{\hat B}_h^t=P_h^t+Q_h^t$, where $Q^t$ was defined in \eqref{eq:z_t} and
\begin{equation*}
  P_h^t\define\alpha_{\hat H}\int_{-\infty}^0\Big((t+h-u)^{\hat H-\frac12}-(t-u)^{\hat H-\frac12}\Big)\,d\hat W_u.
\end{equation*}
\end{enumerate}
\end{example}

Our argument for the quenched ergodic theorem will be based on a two step conditioning procedure making use of an explicit representation of the conditioned process. We state it for a general noise with locally independent increments:

\begin{lemma}\label{lem:conditioning_general}
Let $0\leq s\leq t<t+h$ and $(Z_t)_{t\geq 0}$ be a continuous stochastic process admitting a locally independent decomposition $(\theta_t Z)_{h}=\tilde Z^t_h+\bar Z^t_h$ with respect to $(\F_t)_{t\geq 0}$. Let $X$ and $Y$ be $\F_t$-measurable random variables. 
Then, for any $F:\R^n\to\R$ bounded measurable, 
\begin{equation*}  
  \Expec{F\big(\Psi^X_{s,t+h}(Y, Z)\big)\,\Big|\,\F_t}
  =\Expec{F\big(\Psi_h^{x}(y,\varsigma+\tilde Z^t)\big)}   \bigg|_{\substackal{x&=X,\varsigma=\bar Z^t,\\y&=\Psi^X_{s,t}(Y,Z)}},
\end{equation*}
where $\Psi$ is defined in \cref{def:flow}.
\end{lemma}
\begin{proof} 
  This in an immediate consequence of the flow property of the equation \eqref{eq:ode_solution} and standard properties of conditional expectations.
\end{proof}

Coming back to the flow of the fast motion with frozen slow input \eqref{eq:general_flow-fixed-x}, the following result is an easy consequence of \cref{lem:conditioning_general}:
\begin{lemma}\label{lem:conditioning}
Let $(\F_t)_{t\geq 0}$ be a filtration which is compatible with $\hat{B}$. Fix $0\leq s\leq t<t+h$ and let $X,Y$ be $\F_t$-measurable random variables. Then, for any $F:\R^n\to\R$ bounded measurable,
\begin{equation*}  
  \Expec{F\big(\bar{\Phi}^{X}_{s,t+h}(Y)\big)\,\Big|\,\F_t}=\Expec{F\Big(\Psi_{\frac{h}{\varepsilon}}^{x}\big(y,\varsigma+\sigma\tilde{\hat{B}}\big)\Big)}\bigg|_{\substackal{x&=X,\varsigma=\varepsilon^{-\hat{H}}\sigma\bar{\hat{B}}^{t}_{\varepsilon\cdot},\\y&=\bar{\Phi}^X_{s,t}(Y)}},
\end{equation*}
where $\bar{\hat{B}}^t_{\varepsilon\cdot}\define\big(\bar{\hat{B}}^t_{\varepsilon h}\big)_{h\geq 0}$.
\end{lemma}

We now turn to the fine properties of the smooth part of the increment. For $\alpha>0$ we define the set
\begin{equation}\label{eq:omega}
  \Omega_{\alpha}\define\Big\{f\in \C_0(\R_+,\R^n)\cap\C^2\big((0,\infty),\R^n\big):\limsup_{t\to\infty}\left(t^{\alpha}\big|\dot{f}(t)\big|+t^{1+\alpha}\big|\ddot{f}(t)\big|\right)<\infty\Big\}.
\end{equation}
This space is equipped with the semi-norm 
\begin{equation*}
  \|f\|_{\Omega_\alpha}\define\sup_{t\geq 1}t^{\alpha}\big|\dot{f}(t)\big|+\sup_{t\geq 1}t^{1+\alpha}\big|\ddot{f}(t)\big|.
\end{equation*}
We also set $\Omega_{\alpha-}\define\bigcap_{\beta<\alpha}\Omega_{\beta}$. The motivation for this definition stems from the following lemma:

\begin{lemma}\label{lem:smooth_part_decay}
  Let $\varepsilon>0$ and $t\geq 0$. Then $\varepsilon^{-\hat{H}}\bar{\hat{B}}^t_{\varepsilon\cdot}\overset{d}{=}\bar{\hat{B}}^t\overset{d}{=}\bar{\hat{B}}\in\Omega_{(1-\hat{H})-}$ a.s. and $\|\bar{\hat{B}}\|_{\Omega_\alpha}\in\bigcap_{p\geq 1} L^p$ for any $\alpha<1-\hat{H}$.
\end{lemma}
\begin{proof}
  Let $\delta\in\big(0,1-\hat H\big)$. It is enough to prove that there is a random variable $C>0$ with moments of all orders such that
  \begin{equation}\label{eq:estimate_all_orders}
    \big|\dot{\bar{\hat{B}}}_t\big|\leq \frac{C}{t^{1-\hat{H}-\delta}},\qquad\big|\ddot{\bar{\hat{B}}}_t\big|\leq\frac{C}{t^{2-\hat{H}-\delta}}
  \end{equation}
  for all $t\geq 1$ on a set of probability one. This in turn easily follows from sample path properties of the standard Wiener process. Firstly, we have that
  \begin{equation*}
    \dot{\bar{\hat{B}}}_t=\alpha_{\hat{H}}\left(\hat{H}-\frac12\right)\int_{-\infty}^0(t-u)^{\hat{H}-\frac32}\,dW_u=-\alpha_{\hat{H}}\left(\hat{H}-\frac12\right)\left(\hat{H}-\frac32\right)\int_{-\infty}^0 (t-u)^{\hat{H}-\frac52}W_u\,du
  \end{equation*}
  since $\lim_{u\to-\infty}(t-u)^{\hat{H}-\frac32}W_u=0$. Therefore,
  \begin{align*}
   \big|\dot{\bar{\hat{B}}}_t\big|&\lesssim \left(\sup_{-1\leq s\leq 0} |W_s| \int_{-1}^0 (t-u)^{\hat{H}-\frac52}\,du+\sup_{s\leq -1}\frac{|W_s|}{(t-s)^{\frac12+\delta}}\int_{-\infty}^{-1} (t-u)^{\hat{H}-2+\delta}\,du\right)\\
   &\leq C\left(t^{\hat{H}-\frac52}+(t+1)^{\hat{H}-1+\delta}\right).
  \end{align*}
  The fact that $C$ has moments of all order is an easy consequence of Fernique's theorem. In fact, the Wiener process defines a Gaussian measure on the separable Banach space
  \begin{equation*}
    \mathcal{M}^{\frac12+\delta}\define\left\{f\in\C_0(\R_+,\R^n):\,\|f\|_{\mathcal{M}^{\frac12+\delta}}\define\sup_{u\geq 0}\frac{|f(u)|}{(1+u)^{\frac12+\delta}}<\infty\right\}
  \end{equation*}
  By Fernique's theorem, the random variable $\|W\|_{\mathcal{M}^{\frac12+\delta}}$ has therefore Gaussian tails. The first estimate in \eqref{eq:estimate_all_orders} follows. The bound on $\big|\ddot{\bar{\hat{B}}}_t\big|$ is similar.
\end{proof}

\subsection{A Universal Control}

Let $b\in\S(\kappa,R,\lambda)$, $\varsigma\in\C_0([0,1],\R^n)$, and $u\in L^\infty([0,1],\R^n)$. Let us consider the following controlled ordinary differential equation:
\begin{equation}\label{eq:controlled_ode}
    x^{\varsigma,u}(t)=x_0+\int_0^t b\big(x^{\varsigma,u}(s)\big)\,ds+\varsigma(t)+\int_0^t u(s)\,ds,\qquad t\in[0,1].
\end{equation}
We think of $\varsigma$ as an external `adversary' and of $u$ as a control. Since $b$ is Lipschitz continuous, it is standard that there is a unique global solution to \eqref{eq:controlled_ode}. If $u\equiv 0$, we adopt the shorthand $x^{\varsigma}\define x^{\varsigma,0}$.

The aim of this section is to exhibit an $\eta\in(0,1)$ as large as possible so that the following holds: Given $\bar{R}>0$, there is an $M>0$ such that, for any adversary $\varsigma\in\C_0([0,1],\R^n)$ and any initial condition $x_0\in\R^n$, we can find a control $u\in L^\infty([0,1],\R^n)$ with $|u|_\infty\leq M$ ensuring that the occupation time of $x^{\varsigma,u}$ of the set $\R^n\setminus B_{\bar{R}}$ is at least $\eta$. It is important to emphasize that the sup-norm of the control $|u|_\infty$ may neither depend on the adversary $\varsigma$ nor on the initial condition $x_0$ (otherwise the construction of $u$ essentially becomes trivial). We shall actually choose $u$ as concatenation of the zero function and a \emph{universal} control $\hat u\in L^\infty([0,N^{-1}],\R^n)$ for a sufficiently large, but universal, $N\in\N$.

We begin with a lemma:

\begin{lemma}\label{lem:control_bound}
    There is a constant $C>0$ independent of $\varsigma$ and $u$ such that, for the solution of \eqref{eq:controlled_ode},
    \begin{equation*}
        |x^{\varsigma,u}(t)-x^{\varsigma}(t)|^2\leq C(1+|u|^2_\infty)t
    \end{equation*}
    for all $t\in [0,1]$.
\end{lemma}
\begin{proof}
    Since $b$ is contractive on the large scale, there are constants $D,\tilde{\kappa}>0$ such that
    \begin{equation*}
        \braket{b(x)-b(y),x-y}\leq D-\tilde{\kappa}|x-y|^2
    \end{equation*}
    for all $x,y\in\R^n$, see \cref{rem:large_scall_off_diagonal} \ref{it:off_diagonal}. Define now $f(t)\define e^{\tilde\kappa t}\big|x^{\varsigma,u}(t)-x^{\varsigma}(t)\big|^2$, then
    \begin{equation*}
        f^\prime(t)=\tilde\kappa f(t)+2e^{\tilde\kappa t}\Braket{b\big(x^{\varsigma,u}(t)\big)-b\big(x^{\varsigma}(t)\big)+u(t),x^{\varsigma,u}(t)-x^{\varsigma}(t)}\leq 2D e^{\tilde\kappa}+\frac{|u(t)|^2}{\tilde\kappa}
    \end{equation*}
    for all $t\in [0,1]$. Consequently, setting $C\define \max(2D,\tilde\kappa^{-1})$, we have
    \begin{equation*}
      \big|x^{\varsigma,u}(t)-x^{\varsigma}(t)\big|^2\leq C \int_0^t e^{-\tilde\kappa(t-s)}\left(1+|u(s)|^2\right)\,ds
    \end{equation*}
    and the lemma follows at once.
\end{proof}

For a piecewise constant function $u:[0,1]\to\R^n$, let $\D_u\subset[0,1]$ denote the finite set of discontinuities. We then have the following control result:
\begin{proposition}\label{prop:control}
	Let $\eta<\frac12$ and $\bar{R}>0$. Then there is a value $M>0$ such that the following holds true: For each $\varsigma\in\C_0([0,1],\R^n)$ and each $x_0\in\R^n$, we can find a piecewise constant control $u\in L^\infty([0,1],\R^n)$ with $|u|_\infty+|\D_u|\leq M$ such that the occupation time of $x^{\varsigma,u}$ of the set $\R^n\setminus B_{\bar{R}}$ is greater than or equal to $\eta$.
\end{proposition}

\begin{proof}
  We prove that there exist an integer $N$ and a control $\hat u\in L^\infty([0,N^{-1}])$ with at most two constant pieces independent of both the initial condition $x_0$ and the adversary $\varsigma$ such that either
  \begin{equation*}
    \Leb \Big(\Big\{t\in[0,N^{-1}]:|x^{\varsigma}(t)|>\bar{R} \Big\}\Big)\ge \f \eta N\quad\text{or}\quad\Leb \Big(\Big\{t\in[0,N^{-1}]:|x^{\varsigma,\hat u}(t)|>\bar{R}\Big\}\Big)\ge \f \eta N.
  \end{equation*}
  In the former case, we of course choose $u\equiv 0$, otherwise we let $u=\hat u$. By the flow property of well-posed ordinary differential equations, the solution to \eqref{eq:controlled_ode} restarted at time $N^{-1}$ solves a similar equation (with new adversary $\tilde{\varsigma}(\cdot)=\theta_{N^{-1}} \varsigma \in\C_0([0,1-N^{-1}],\R^n)$ and initial condition $x^{\varsigma,u}(N^{-1})$). Upon constructing $\hat u$, we can thus easily deduce the proposition by iterating this construction.

  Suppose that the time spent by uncontrolled solution $(x_t^\varsigma)_{t \in [0,N^{-1}]}$ in $\R^n\setminus B_{\bar{R}}$ is strictly less than $\frac{\eta}{N}$. We let $A_{x_0,\varsigma}$ be the set of times $t\in[0,N^{-1}]$ at which $|x^\varsigma(t)|\leq \bar{R}$. Note that $A_{x_0,\varsigma}$ is the union of a countable number of closed, disjoint intervals. By assumption, we have $\Leb(A_{x_0,\varsigma})>(1-\eta)N^{-1}$. 
  
  For $\delta\define (2N)^{-1}$ and $e$ any fixed unit vector, we define $\hat u$ to be the piecewise constant function
  \begin{equation*}
    \hat u(t)=\begin{cases}
     \frac{2\bar{R}+1}{(1-2\eta)\delta}e, & t\in [0,  \delta],\\
    -\frac{2\bar{R}+1}{(1-2\eta)\delta}e, & t\in (\delta, 2\delta],
    \end{cases}
  \end{equation*}
  so that
  \begin{equation*}
    \int_0^t \hat u(s)\,ds=\begin{cases}
     \frac{2\bar{R}+1}{(1-2\eta)\delta}te,   & t\in [0,  \delta],\\
    \frac{2\bar{R}+1}{(1-2\eta)\delta}(2\delta-t)e, & t\in (\delta, 2\delta].
    \end{cases}
  \end{equation*}
  We observe that
  \begin{equation}\label{eq:lower_control}
    |x^{\varsigma,\hat u}(t)|\geq \left  |\int_0^t \hat u(s)\,ds \right|-|x^\varsigma(t)|-\Lip{b}\int_{0}^t\big|x^{\varsigma,\hat u}(s)-x^\varsigma(s)\big|\,ds.
  \end{equation}
  Moreover, owing to \cref{lem:control_bound}, we can bound
  \begin{equation}\label{eq:lower_control_n}
    \phantom{\leq}\int_{0}^t\big|x^{\varsigma,\hat u}(s)-x^\varsigma(s)\big|\,ds\leq\sqrt{C}(1+|\hat u|_\infty)\int_{0}^{2\delta}\sqrt{s}\,ds=\frac{2\sqrt{C}}{3 N^{\frac32}}\left(1+\frac{2(2\bar{R}+1)N}{1-2\eta}\right)<\Lip{b}^{-1},
  \end{equation}
  provided we choose the integer $N=N(C,\bar{R},\eta,\Lip{b})$ large enough. Define the set $B_{x_0,\varsigma}\define A_{x_0,\varsigma}\cap [(1-2\eta)\delta,(1+2\eta)\delta]$. Combining \eqref{eq:lower_control} and \eqref{eq:lower_control_n}, we then certainly have that $|x^{\varsigma,\hat u}(t)|>\bar{R}$ for all $t\in B_{x_0,\varsigma}$. Since
  \begin{equation*}
    \Leb(B_{x_0,\varsigma})\geq\frac{(1-\eta)}{N}-2(1-2\eta)\delta=\frac{\eta}{N}
  \end{equation*}
  and $|\hat u|_\infty$ as well as $|\D_{\hat{u}}|$ only depend on $N$ and $\bar R$,  this finishes the proof.
\end{proof}

We conclude our study of the deterministic controlled ODE \eqref{eq:controlled_ode} with the following stability result which is proven by a standard Gr\"onwall argument:
\begin{lemma}\label{lem:cont_control}
    Let $x^{\varsigma,u}$ denote the solution to the controlled differential equation \eqref{eq:controlled_ode} with initial condition $x_0\in\R^n$ and control $u\in L^\infty([0,1],\R^n)$. Then, for any $w\in\C_0([0,1],\R^n)$, we have the bound
    \begin{equation*}
        |x^{\varsigma,u}-\tilde{x}|_\infty\leq e^{\Lip{b}}\left|\int_0^\cdot u(s)\,ds-w\right|_\infty,
    \end{equation*}
    where $\tilde{x}$ is the unique solution to
    \begin{equation*}
      \tilde{x}(t)=x_0+\int_0^t b\big(\tilde{x}(s)\big)\,ds+w(t)+\varsigma(t),\qquad t\in[0,1].
    \end{equation*}
\end{lemma}

\subsection{Exponential Stability of the Conditional Evolution}
We now turn to the conditional evolution of \eqref{eq:general_flow-fixed-x} derived in \cref{lem:conditioning}. For brevity, we drop the hat on the driving fBm throughout this and the next section. Remember that we have to study SDEs driven by a Riemann-Liouville process 
\begin{equation*}
\tilde{B}_t\define\alpha_H\int_0^{t} (t-u)^{H-\frac12}\,dW_u,
\end{equation*}
where $(W_t)_{t\geq 0}$ is a standard Wiener process. Recall from \cref{def:flow} that, for $\varsigma\in\C_0(\R_+,\R^n)$, $\Psi_{s,t}(\cdot, \vsigma+\sigma\tilde B)$ denotes the solution flow to the equation 
\begin{equation}\label{eq:rl_sde}
  dX_t=b(X_t)\,dt+d\varsigma_t+\sigma\,d\tilde{B}_t.
\end{equation}
For brevity, let us henceforth set $\Psi_{s,t}^{\varsigma}(\cdot)\define\Psi_{s,t}(\cdot,\varsigma+\sigma\tilde B)$.

We first prove that---starting from any two initial points---the laws of the solutions converge to each other with an exponential rate. This however does not yet imply the convergence of $\L\big(\Psi_t^{\varsigma}(x)\big)$ to the first marginal of the invariant measure $\pi$ of the equation $dX_t=b(X_t)\,dt+\sigma\,dB_t$ since, even if we choose $X_0\sim\pi$, we have $\L\big(\Psi_t^{\varsigma}(X_0)\big)\neq\pi$ for $t>0$ in general. 

As a preparation, we let $\big(\C_0([0,1],\R^n),\cH_H,\mu_H\big)$ denote the abstract Wiener space induced by the Gaussian process $(\tilde B_t)_{t\in[0,1]}$. Recall that the Cameron-Martin space is given by $\cH_H=\mathscr{K}_H(H_0^1)$, where
\begin{equation*}
  \mathscr{K}_H f(t)\define\begin{cases}
  \displaystyle\alpha_H\int_0^t (t-s)^{H-\frac32}f(s)\,ds, & H>\frac12,\\ 
  \displaystyle\alpha_H\frac{d}{dt}\int_0^t (t-s)^{H-\frac12}f(s)\,ds, & H<\frac12,
  \end{cases}\qquad t\in[0,1],
\end{equation*}
and 
\begin{equation*}
  H_0^1\define\left\{f=\int_0^\cdot\dot{f}(s)\,ds:\,\dot{f}\in L^2([0,1],\R^n)\right\}
\end{equation*}
is the Cameron-Martin space of the standard Wiener process. The inner product on $\cH_H$ is defined by $\braket{\mathscr{K}_H f,\mathscr{K}_H g}_{\cH_H}\define\braket{\dot{f},\dot{g}}_{L^2}$. 

We shall make use of the following simple observation:
\begin{lemma}\label{lem:cameron_martin_facts}
  Let $f:[0,1]\to\R^n$ be piecewise linear with $f(0)=0$. Then, for each $H\in(0,1)$, $f\in\cH_H$ and
  \begin{equation}\label{eq:cameron_martin_bound}
    \|f\|_{\cH_H}\lesssim|\dot{f}|_\infty \big(1+\big|\D_{\dot{f}}\big|\big).
  \end{equation}
\end{lemma}
\begin{proof}
  It follows from \cite[Theorem 5]{Picard2011} (see also \cite{Samko1993}) that the inverse of $\K_H$ exists on the set of Lipschitz functions and there is a numerical constant $\varrho_H>0$ such that $\K_H^{-1}=\varrho_H\K_{1-H}$. Notice also that we have $\frac{d}{dt}\K_H^{-1}f=\K_H^{-1}\dot{f}$.
  
  Let us first consider the case $H<\frac12$. The bound \eqref{eq:cameron_martin_bound} is an immediate consequence of 
  \begin{equation*}
    \left|\frac{d}{dt}\K_H^{-1} f(t)\right|\leq\varrho_H\int_0^t (t-s)^{-H-\frac12}\big|\dot{f}(s)\big|\,ds\lesssim|\dot{f}|_\infty \qquad\forall\,t\in[0,1].
  \end{equation*}
  For $H>\frac12$ we let $\tau_1,\dots,\tau_k$ denote the jump points of $\dot{f}$ in the interval $[0,t)$. Notice that
  \begin{align*}
    \left|\frac{d}{dt}\K_H^{-1} f(t)\right|&\leq\varrho_H\left|\frac{d}{dt}\left(\sum_{i=1}^{k-1}\int_0^{\tau_1}(t-s)^{\frac12-H}\dot{f}(s)\,ds+\cdots+\int_{\tau_k}^t (t-s)^{\frac12-H}\dot{f}(s)\,ds\right)\right| \\
    &\lesssim |\dot{f}|_\infty\big(1+|\D_{\dot{f}}|\big)t^{\frac12-H}.
  \end{align*}
  Since $1-2H>-1$, we obtain
  \begin{equation*}
    \|f\|_{\cH_H}=\left\|\frac{d}{dt}\K_H^{-1}f\right\|_{L^2}\lesssim|\dot{f}|_\infty\big(1+|\D_{\dot{f}}|\big),
  \end{equation*}
  as required.
  
\end{proof}

The next important lemma lifts the control result of \cref{prop:control} to solutions of SDEs with additive noise:
\begin{lemma}\label{lem:probabilistic_control}
  Let $b\in\S(\kappa, R,\lambda)$ and $\sigma\in\Lin{n}$ be invertible. Then, for any $\bar R>0$ and any $\eta\in(0,\frac12)$, there is constant $\a_{\eta,\bar{R}}>0$ such that the following holds: For each $x\in\R^n$ and each $\varsigma\in\C_0(\R_+,\R^n)$, we can find an event $\A_{x,\varsigma}$ with $\prob(\A_{x,\varsigma})\geq\a_{\eta,\bar{R}}$ such that
  \begin{equation*}
    \int_0^1 \1_{\big\{t: \big|\Psi_{t}^{\varsigma}(x)(\omega)\big|>\bar R\big\}}(s)\,ds > \eta \qquad  \forall\, \omega \in \A_{x,\varsigma}.
  \end{equation*}
\end{lemma}

\begin{proof}
  Let $u_{x,\varsigma}\in L^\infty([0,1],\R^n)$ be the piecewise constant control furnished by \cref{prop:control} such that the occupation time of $x^{\varsigma,u_{x,\varsigma}}$ of the set $\R^n\setminus B_{\bar{R}+1}$ is greater than $\eta$. We set $U_{x,\varsigma}\define\int_0^\cdot u_{x,\varsigma}(s)\,ds$ and note that $U_{x,\varsigma}$ is piecewise linear. \Cref{lem:cont_control} allows us to choose an $\varepsilon>0$ (independent of $x$ and $\varsigma$) such that, on the event $\A_{x,\varsigma}\define\big\{\big|U_{x,\varsigma}-\sigma\tilde{B}\big|_\infty\leq\varepsilon\big\}$, the occupation time of $\big(\Psi^{\vsigma}_{h}(x)\big)_{h\in[0,1]}$ of $\R^n\setminus B_{\bar R} $ exceeds $\eta$. 

  It remains to show that $\inf_{x,\varsigma}\prob(\A_{x,\varsigma})>0$. To this end, we first note that $U_{x,\varsigma}\in\cH_H$ by \cref{lem:cameron_martin_facts}. By the Cameron-Martin formula (see e.g. \cite{Bogachev1998}),  
  \begin{align*}
    \prob(\A_{x,\varsigma})&\geq\prob\big(\big|\sigma^{-1}U_{x,\varsigma}-\tilde{B}\big|_\infty\leq|\sigma|^{-1}\varepsilon\big)\\
    &=\exp\left(-\frac12\|\sigma^{-1}U_{x,\varsigma}\|_{\cH_H}^2\right)\int_{\{|x|_\infty\leq|\sigma|^{-1}\varepsilon\}}e^{\braket{x,U_{x,\varsigma}}_{\cH_H}}\,\mu_H(dx).
  \end{align*}
  Consequently, Jensen's inequality and spherical symmetry give
  \begin{equation}\label{eq:quant_lower_bound}
    \prob(\A_{x,\varsigma})\geq\exp\left(-\frac12\|\sigma^{-1}U_{x,\varsigma}\|_{\cH_H}^2\right)\prob\big(|\tilde{B}|_\infty\leq|\sigma|^{-1}\varepsilon\big).
  \end{equation}
  Combining \cref{prop:control,lem:cameron_martin_facts}, we obtain that $\sup_{x,\varsigma}\|U_{x,\varsigma}\|_{\cH_H}\lesssim M(1+M)$. This concludes the proof.
\end{proof}

\begin{proposition}\label{prop:conditional_initial_condition_wasserstein}
 Let $\sigma\in\Lin{n}$ be invertible. Then, for any $\kappa, R>0$ and any $p\geq 1$, there exists a number $\Lambda=\Lambda(\kappa,R,p)\in(0,\kappa)$ such that the following holds: If $b\in\S\big(\kappa,R, \Lambda\big)$, there are constants $c,C>0$ such that, for any $\varsigma\in\C_0([0,1],\R^n)$,
  \begin{equation*}
    \W^p\Big(  \L\big( {\Psi}^{\vsigma}_t(Y)\big),  \L\big( {\Psi}^{\vsigma}_t(\tilde Y)\big)\Big)\leq C\W^p\big(\L(Y),\L(\tilde Y)\big) e^{-c t}
  \end{equation*}
   for all $t\geq 0$.
 \end{proposition}

\begin{proof} 
Write $X_t\define\Psi^{\vsigma}_{t}(Y)$ and $Z_t=\Psi_{t}^{\vsigma}(\tilde Y)$. Let $\mu_t\define\L(X_t)$ and $\nu_t\define\L(Z_t)$, thus $(X_t,Z_t)$ is a synchronous coupling of $\mu_t$ and $\nu_t$. Our strategy for proving the exponential convergence of $t\mapsto\W^p(\mu_t,\nu_t)$ is to show that, for any $t>0$, the evolution of $(X_s)_{s\in[t,t+1]}$ conditional on $\F_t$ spends a sufficient amount of time in the contractive region $\{|x|>R\}$. As noted in \cref{example-1} \ref{it:rough_decomposition}, there is an independent increment decomposition $(\theta_t\tilde{B})_{h}= Q^t_h+\tilde{B}^t_h$ for the Riemann-Liouville process. Using this and the conditional evolution derived in \cref{lem:conditioning_general}, we find
  \begin{align}
    \Expec{\big|X_{t+1}-Z_{t+1}\big|^p}&    =  \Expec{\Expec{\big|\Psi^{\vsigma}_{t,t+1}(X_t)-\Psi^{\vsigma}_{t,t+1}(Z_t)\big|^p\,\middle|\,\F_t}}\nonumber\\
    &=     \Expec{\Expec{\Big|\Psi_{1}\big(X_t, \theta_t\vsigma+\sigma\theta_t\tilde B\big)-\Psi_{1}\big(Z_t, \theta_t\vsigma+\sigma \theta_t\tilde B\big)\Big|^p\,\middle|\,\F_t}}\nonumber\\
    &=\Expec{\Expec{\Big|\Psi_{1}\big(X_t, \theta_t\vsigma+\sigma Q^t+\sigma\tilde{B}^t\big)-\Psi_{1}\big(Z_t, \theta_t\vsigma+\sigma Q^t+\sigma\tilde{B}^t\big)\Big|^p\,\middle|\,\F_t}} \nonumber\\
       &=     \Expec{  \Expec{\Big|\Psi^{\theta_t\vsigma+\ell}_{1} (x)-\Psi_{1}^{\theta_t\vsigma+\ell}(z)\Big|^p}\bigg|_{\substackal{x&=X_t,z=Z_t,\\\ell&=\sigma Q^t}}},\label{two-initial-conditions}
 \end{align}
 where in the last step we also used that $(\tilde{B}^t_{h})_{h\geq 0}\overset{d}{=}(\tilde{B}_{h})_{h\geq 0}$.

By assumption, the drift $b$ does not expand by more than a factor of $\Lambda$ on all of $\R^n$. We therefore have the pathwise estimate
\begin{equation}\label{eq:rl_lipschitz}
\big|\Psi_{s,t}^{\theta_t\vsigma+\ell}(x)-\Psi_{s,t}^{\theta_t\vsigma+\ell}(z)\big|^p\le e^{p(t-s)\Lambda} |x-z|^p
\end{equation}
for all $0\leq s< t\leq 1$. Let $\eta\in(0,\frac12)$ and $\bar\kappa\in(0,\kappa)$ be such that $\Xi\define\bar\kappa\eta-\Lambda(1-\eta)>0$ (recall that we assume $\Lambda<\kappa$). Let $\bar R>R$ be the corresponding radius furnished by \cref{lem:bigger_ball}. For any $x\in\R^n$ and any $\varsigma,\ell\in\C_0(\R_+,\R^n)$, let $\A_{x,\theta_t\varsigma+\ell}$ be the event from \cref{lem:probabilistic_control}. Recall that $\prob(\A_{x,\theta_t\varsigma+\ell})\geq\a_{\eta,\bar{R}}>0$ and 
\begin{equation*}
  \int_0^1 \1_{\big\{s: \big|\Psi_{s}^{\theta_t\vsigma+\ell}(x)(\omega)\big|>\bar R\big\}}(r)\,dr > \eta \qquad  \forall\, \omega \in \A_{x,\theta_t\varsigma+\ell}.
\end{equation*}
Since $\Xi>0$, by possibly decreasing $\Lambda$ we can also ensure that 
\begin{equation}\label{eq:lambda}
  0<\Lambda < \frac1p\log\left(\f  {1-\a_{\eta,\bar{R}}e^{-p \Xi} }{1-\a_{\eta,\bar{R}}}\right).
\end{equation}
Owing to pathwise continuity of $h\mapsto\Psi_h^{\theta_t\vsigma+\ell}(x)$, there are random times $t_1,\dots,t_{2N(\omega)}$ such that, for all $\omega\in \A_{x,\theta_t\varsigma+\ell}$,
\begin{itemize}
  \item $0\leq t_1(\omega)<\cdots<t_{2N(\omega)}(\omega)\leq 1$, 
  \item $\displaystyle\sum_{i=1}^{N(\omega)}\big(t_{2i}(\omega)-t_{2i-1}(\omega)\big)\geq\eta$, and
  \item $\displaystyle\bigcup_{i=1}^{N(\omega)}\big(t_{2i-1}(\omega),t_{2i}(\omega)\big)\subset\big\{h\in[0,1]:\,\big|\Psi_{h}^{\theta_t\varsigma+\ell}(x)(\omega)\big|>\bar R\big\}$.
\end{itemize}
Together with \eqref{eq:rl_lipschitz} it follows that, on the event $\A_{x,\theta_t\varsigma+\ell}$,
\begin{align*}
&\phantom{\le}\big|\Psi_{1}^{\theta_t\varsigma+\ell}(x)-\Psi_{1}^{\theta_t\varsigma+\ell}(z)\big|^p
= \Big|\Psi_{t_{2N},1}^{\theta_t\varsigma+\ell}\Big(\Psi_{t_{2N}}^{\theta_t\varsigma+\ell} (x)\Big)-\Psi_{t_{2N},1}^{\theta_t\varsigma+\ell}\Big(\Psi_{t_{2N}}^{\theta_t\varsigma+\ell} (z)\Big)\Big|^p\\
&\le e^{p(1-t_{2N})\Lambda} \big|\Psi_{t_{2N}}^{\theta_t\varsigma+\ell}(x)-\Psi_{t_{2N}}^{\theta_t\varsigma+\ell}(z)\big|^p\\
&\le e^{p (1-t_{2N})\Lambda} e^{-p (t_{2N}-t_{2N-1})\bar\kappa} \big|\Psi_{t_{2N-1}}^{\theta_t\varsigma+\ell}(x)-\Psi_{t_{2N-1}}^{\theta_t\varsigma+\ell}(z)\big|^p\\
&\leq\cdots\le \exp\left[p\left(\Lambda\sum_{i=0}^N(t_{2i+1}-t_{2i})-\bar\kappa\sum_{i=0}^{N}  (t_{2i}-t_{2i-1}) \right)\right] |x-z|^p\\
&\le e^{ -p\Xi}|x-z|^p,
\end{align*}
where we have set $t_{2N+1}\define 1$ for convenience. On the complementary event $\Omega\setminus \A_{x,\theta_t\varsigma+\ell}$, we apply the trivial estimate \eqref{eq:rl_lipschitz}. Inserting these bounds back into \eqref{two-initial-conditions}, we conclude that
\begin{equation*}
  \Expec{\big|X_{t+1}-Z_{t+1}\big|^p}\leq\Big(\big(1-\a_{\eta,\bar R}\big)e^{p\Lambda}+\a_{\eta,\bar R}e^{-p\Xi}\Big)\Expec{|X_t-Z_t|^p}\define\rho\Expec{|X_t-Z_t|^p}.
\end{equation*}
Observe that $\rho<1$ by \eqref{eq:lambda}. Finally, a straight-forward induction shows that
\begin{equation}\label{eq:to_minimize}
  \W^p\Big(\L\big( {\Psi}^{\vsigma}_t(Y)\big),  \L\big( {\Psi}^{\vsigma}_t(\tilde Y)\big)\Big)\leq
  \big\|X_t-Z_t\big\|_{L^p}\leq e^{\Lambda} \rho^{[t]}\big\|Y-\tilde Y\big\|_{L^p}\leq\frac{e^{\Lambda}}{\rho}e^{-|\log\rho|t}\big\|Y-\tilde Y\big\|_{L^p},
\end{equation}
where $[\cdot]$ denotes the integer part. Minimize over the set of couplings of $\L(Y)$ and $\L(\tilde{Y})$ to conclude the proof.
\end{proof}

A more explicit expression for the threshold value $\Lambda(\kappa,R,p)$ can be derived by the method outlined in \cref{rem:constant_xi} below. We abstain from including further details in this work. Let us however introduce the following notation:
\begin{definition}
  Let $\kappa,R>0$ and $p\geq 1$. We abbreviate $\S_p(\kappa,R)\define\S\big(\kappa,R,\Lambda(\kappa,R,p)\big)$ with the constant from \cref{prop:conditional_initial_condition_wasserstein}.
\end{definition}

By \cref{lem:conditioning}, the Wasserstein bound of \cref{prop:conditional_initial_condition_wasserstein} lifts to bounds on the fast motion with frozen slow input \eqref{eq:general_flow-fixed-x}. We obtain the following Lipschitz dependence of the flow $\bar \Phi$ on the initial value:

\begin{corollary}\label{cor:fast_different_initial}
  Let $(\F_t)_{t\geq 0}$ be a filtration compatible with the fBm $\hat{B}$. Let $0\leq s\leq t$ and let $X$, $Y$, and $\tilde{Y}$ be $\F_s$-measurable random variables. Suppose that there are $\kappa,R>0$ such that $b(x,\cdot)\in\S_1(\kappa,R)$ for every $x\in\R^d$. Then there is a constant $c>0$ such that, for any Lipschitz continuous function $h:\R^d\times\R^n\to\R$,
 \begin{equation*}
    \Big|\Expec{h\big(X,\bar{\Phi}_{s,t}^X(Y)\big)-h\big(X,\bar{\Phi}_{s,t}^X(\tilde{Y})\big)\,\middle|\,\F_s}\Big|\lesssim\Lip{h}|Y-\tilde Y|e^{-c\frac{|t-s|}{\varepsilon}}.
  \end{equation*}
  If, in addition, $b(x,\cdot)\in\S_p(\kappa,R)$ for all $x\in\R^d$, then also
  \begin{equation*}
    \Big\|\bar{\Phi}_{s,t}^X(Y)-\bar{\Phi}_{s,t}^X(\tilde{Y})\Big\|_{L^p}\lesssim\|Y-\tilde Y\|_{L^p}e^{-c\frac{|t-s|}{\varepsilon}}.
  \end{equation*}
\end{corollary}
\begin{proof}
  The first estimate is an immediate consequence of \cref{lem:conditioning} and Kantorovich-Rubinstein duality. The second bound follows from the fact that we used a synchronous coupling in the proof of \cref{prop:conditional_initial_condition_wasserstein}.
\end{proof}

The proof of \cref{prop:conditional_initial_condition_wasserstein} shows that its conclusion actually holds if $\tilde B$ is replaced by another process $Z$ with similar properties:
\begin{remark}\label{rem:Wasserstein-general}
Let $Z$ be a process with locally independent increment decomposition $\theta_t Z=\bar Z^t+\tilde Z^t$. Assume that
\begin{enumerate}
\item the $\F_t$-adapted part $\bar Z^t$ takes values in $\C_0(\R_+,\R^n)$ and
\item\label{it:cm_dense} there is a unit vector $e\in\R^n$ such that, for each $t\geq 0$, $\L\big((\tilde Z^t_h\cdot e)_{h\in[0,1]}\big)$ is supported on all of $\C_0([0,1])$.
\end{enumerate} 
Then a statement similar to \cref{prop:conditional_initial_condition_wasserstein} holds.
\end{remark}

\begin{example}\label{ex:titmarsh}
  Suppose that $\tilde{Z}^t_h=\int_t^{t+h}\mathfrak{G}(t+h-s)\,dW_s$ for some kernel $\mathfrak{G}:\R_+\to\Lin{n}$ which is square integrable at the origin and continuous on $(0,\infty)$. Then the requirement \ref{it:cm_dense} in \cref{rem:Wasserstein-general} holds if $\int_0^t|\mathfrak{G}(s)|\,ds>0$ for each $t>0$. Indeed, this can be shown by a clever application of Titmarsh's convolution theorem as in \cite[Lemma 2.1]{Cherny2008}.
\end{example}

The example shows that in particular an fBm of any Hurst parameter $H\in(0,1)$ falls in the regime of \cref{rem:Wasserstein-general}. Hence, we have the following corollary to \cref{prop:conditional_initial_condition_wasserstein}:
\begin{corollary}\label{conveergence-equilibrium}
 Let $p\geq 1$ and suppose that $b\in\S_p(\kappa,R)$ for some $\kappa,R>0$. Let $(X_t)_{t\geq 0}$ be the solution to 
 \begin{equation}\label{eq:fbm_sde}
   dX_t=b(X_t)\,dt+\sigma\,dB_t
 \end{equation}
 started in the generalized initial condition $\mu$, where $(B_t)_{t\geq 0}$ is an fBm with Hurst parameter $H\in(0,1)$ and $\sigma\in\Lin{n}$ is invertible. Then there is a unique invariant measure $\I_\pi\in\P(\R^n\times\H_H)$ for the equation \eqref{eq:fbm_sde} in the sense of \cref{initial-condition}. Moreover, writing $\pi=\Pi_{\R^n}^*\I_\pi$ for the first marginal, there are constants $c,C>0$ such that
 \begin{equation}\label{eq:wasserstein_fbm}
    \W^p\big(\L( X_t),\pi\big)\leq C\bW^p(\mu,\I_\pi) e^{-ct}
  \end{equation}
  for all $t\geq 0$.
 \end{corollary}
\begin{proof}
  By \cref{prop:existence_invariant_measure}, we know that there is an invariant measure $\I_\pi$ to \eqref{eq:fbm_sde} with moments of all orders. The Wasserstein estimate \eqref{eq:wasserstein_fbm} then follows by the very same arguments as in \cref{prop:conditional_initial_condition_wasserstein}. The only difference is that we now have to specify a generalized initial condition $\nu\in\P\big((\R^n\times\H_H)^2\big)$ for the coupling $(X_t,Z_t)$, see \cref{sec:physical_solution}. Unlike for the conditioned dynamics, we have $Z_t\sim\pi$ if we start $Z$ in the invariant measure $\I_\pi$. In order for our previous argument to apply, we need to ensure that the past of the noises in the synchronous coupling coincide. In \eqref{eq:to_minimize} we can thus only minimize over couplings in the set 
  \begin{equation*}
    \big\{\rho\in\P\big((\R^n\times\H_H)^2\big):\,\rho(\R^n\times\R^n\times\Delta_{\H_H})=1\big\},
  \end{equation*}
  which precisely yields \eqref{eq:wasserstein_fbm}.
\end{proof}

\subsection{Quenched Convergence to the Invariant Measure}\label{quenched-convergence}

The other distance, which will play a r\^ole in \cref{sec:uniform_bounds} below, is between $\L\big(\Psi_t^\vsigma(Y)\big)$ and the stationary law $\pi$ of the equation \eqref{eq:fbm_sde}. We stress that---contrarily to the proof of \cref{conveergence-equilibrium}---we cannot simply start the process in the invariant measure. In fact, the measure $\pi$ is not stationary for \eqref{eq:rl_sde} since the increments of $\tilde{B}$ are not stationary. It is therefore necessary to wait for a sufficient decay of the deterministic `adversary' $\varsigma$, whence we only find an algebraic rate of convergence. Before we state the result, let us first illustrate that there is indeed no hope for an exponential rate:
\begin{example}
  Let
  \begin{equation*}
    dX_t=-X_t\,dt+d\tilde{B}_t,\qquad dY_t=-Y_t\,dt+dB_t.    
  \end{equation*}
  If we start both $X$ and $Y$ in the generalized initial condition $\delta_0\otimes\sW$, then $\L(X_t)=N(0,\Sigma_{t}^2)$ and $\L(Y_t)=N(0,\bar\Sigma_{t}^2)$ where
  \begin{equation*}
    \Sigma_t^2=\bar\Sigma_t^2-\Expec{\left|\int_0^te^{-(t-s)}\dot{\bar{B}}_s\,ds\right|^2}.
  \end{equation*}
  In particular, $\W^2\big(\L(X_t),\L(Y_t)\big)=|\Sigma_{t}-\bar\Sigma_t|\gtrsim t^{-(1-\hat{H})}$ uniformly in $t\geq 1$. Since it is easy to see that $\W^2\big(\L(Y_t),\pi\big)\lesssim e^{-t}$, it follows that $\W^2\big(\L(X_t),\pi\big)\gtrsim t^{-(1-\hat{H})}$.
\end{example}

\begin{proposition}\label{prop:conditional_stationary_wasserstein}
  Suppose that $b\in\S_p(\kappa,R)$ for some $\kappa,R>0$ and $\sigma\in\Lin{n}$ is invertible. Let $p\geq 1$, $\varsigma\in\Omega_\alpha$ for some $\alpha>0$, and $Y$ be an $\F_0$-measurable random variable. Then, for each $\beta<\min\big(\alpha,1-H\big)$, there is a constant $C>0$ such that
  \begin{equation}\label{eq:wasserstein_quenched}
    \W^p\big(\L(\Psi^\vsigma_t(Y)), \pi\big)\leq C\frac{\big(1+\|\varsigma\|_{\Omega_\beta}\big)\big(1+\W^p(\L(Y),\pi)\big)}{t^{\beta}}
  \end{equation}
  for all $t>0$.
\end{proposition}

\begin{proof}
  Fix $t\geq 1$, abbreviate $X\define\Psi_\cdot^{\vsigma}(Y)$, and let $Z$ be the stationary solution to the equation \eqref{eq:fbm_sde}. We assume that $X$ and $Z$ are driven by the same Wiener process. Let us first consider the case $p\geq 2$. Recall the following locally independent decompositions from \cref{sec:increment}:
    $$\theta_t B=\bar B^t+\tilde B^t, \qquad \theta_t \tilde B=Q^t+\tilde B^t.$$
  Remember also that the `smooth' part of the fBm increment can be further decomposed as $\bar{B}^t=P^t+Q^t$, see \cref{example-1} \ref{it:smooth_decomposition}. Therefore, 
  \begin{align}
    \Expec{\big|X_{t+1}-Z_{t+1}\big|^p}&    =  \Expec{\Expec{\big|\Psi_{t,t+1}(X_t,\vsigma+ \sigma\tilde B)-\Psi_{t,t+1}(Z_t,\sigma B)\big|^p\,\middle|\,\F_t}}\nonumber\\
        &=\Expec{\Expec{\Big|\Psi_{1}\big(X_t, \theta_t\vsigma+\sigma Q^t+\sigma \tilde B^t\big)-\Psi_{1}\big(Z_t,\sigma P^t+\sigma Q^t+\sigma\tilde B^t\big)\Big|^p\,\middle|\,\F_t}}\nonumber\\
       &=\Expec{\Big|  \Psi_1^{\theta_t \vsigma+ \ell} (x) - \Psi_1^{\bar\ell+\ell} (z) \Big|^p\bigg|_{\substackal{x&=X_t,z=Z_t,\\\ell&=\sigma Q^t,\bar{\ell}=\sigma P^t}}}\label{eq:expec_diff}
  \end{align}
  Write $R_h\define\Psi_h^{\theta_t \vsigma+\ell} (x)$ and $S_h\define \Psi_h^{\bar\ell+\ell} (z) $. Notice that, since $\vsigma$ and $\bar\ell$ are differentiable,
  \begin{align*}
    \frac{d}{dh}\big|R_{h}-S_{h}\big|^p&=p\Braket{\dot{\varsigma}_{t+h}-\dot{\bar{\ell}}_h+b\big(R_{h}\big)-b\big(S_{h}\big),R_{h}-S_{h}}\big|R_{h}-S_{h}\big|^{p-2}\\
    &\leq p(\Lambda+\gamma)\big|R_{h}-S_{h}\big|^p
     +\left(\frac{p-1}{\gamma p}\right)^{p-1}\left(|\dot{\varsigma}_{t+h}|+|\dot{\bar{\ell}}_{h}|\right)^p
  \end{align*}
  for any $\gamma>0$, where $\Lambda=\Lambda(\kappa,R,p)$ is the expansion threshold derived in \cref{prop:conditional_initial_condition_wasserstein}. It follows that, for any $0\leq h_1\leq h_2\leq 1$,
  \begin{align}
    &\phantom{\leq}\big|R_{h_2}-S_{h_2}\big|^p\nonumber\\
    &\leq\big|R_{h_1}-S_{h_1}\big|^p e^{p(\Lambda+\gamma)(h_2-h_1)}+\left(\frac{p-1}{\gamma p}\right)^{p-1}\int_{h_1}^{h_2}e^{p(\Lambda+\gamma)(h_2-s)}\left(|\dot\varsigma_{t+s}|+|\dot{\bar{\ell}}_{s}|\right)^p\,ds\nonumber\\
    &\leq \big|R_{h_1}-S_{h_1}\big|^p e^{p(\Lambda+\gamma)(h_2-h_1)}+C_\gamma(h_2-h_1),\label{eq:estimate_waserstein_1}
  \end{align}
  where we abbreviated
  \begin{equation*}
    C_\gamma\define \left(\frac{p-1}{\gamma p}\right)^{p-1}\left(\frac{\|\varsigma\|_{\Omega_\beta}}{t^{\beta}}+|\dot{\bar{\ell}}|_\infty\right)^p.
  \end{equation*}
  We now argue similarly to \cref{prop:conditional_initial_condition_wasserstein}: Pick $\eta\in(0,\frac12)$ and $\bar{\kappa}\in(0,\kappa)$ such that $\Xi\define \eta\bar{\kappa}-(1-\eta)\Lambda>0$.  Let $\bar{R}>0$ be the corresponding constant of \cref{lem:bigger_ball} and $\A_{x,\theta_t\varsigma+\ell}$ be the event furnished by \cref{lem:probabilistic_control}. As before, we write $t_1,\dots,t_{2N(\omega)}$ for the random times characterizing the excursions of $(R_h)_{h\in[0,1]}$ outside of $B_{\bar R}$, see \cref{prop:conditional_initial_condition_wasserstein}. By an argument similar to \eqref{eq:estimate_waserstein_1},  
  \begin{equation}\label{eq:estimate_waserstein_2}
    \big|R_{t_{2i}}-S_{t_{2i}}\big|^p\leq \big|R_{t_{2i-1}}-S_{t_{2i-1}}\big|^p e^{p(\gamma-\bar{\kappa})(t_{2i}-t_{2i-1})}+C_\gamma (t_{2i}-t_{2i-1})
  \end{equation}
  for all $i=1,\dots,N(\omega)$ on the set $\A_{x,\theta_t\varsigma+\ell}$. Combining \eqref{eq:estimate_waserstein_1} and \eqref{eq:estimate_waserstein_2}, we further find on this set
  \begin{align*}
    \phantom{\leq}\big|R_{1}-S_{1}\big|^p&\leq e^{p(\Lambda+\gamma)(1-t_{2k})}\big|R_{t_{2k}}-S_{t_{2k}}\big|^p+C_\gamma(1-t_{2k})\\
    &\leq e^{p(\Lambda+\gamma)(1-t_{2k})}e^{p(\gamma-\bar{\kappa})(t_{2k}-t_{2k-1})}\big|R_{t_{2k-1}}-S_{t_{2k-1}}\big|^p+C_\gamma(1-t_{2k-1})\\
    &\leq\cdots \leq e^{p(\Lambda+\gamma)(1-\eta) +p(\gamma-\bar \kappa)\eta}|x-z|^p+C_\gamma 
    \leq e^{-p\big(\Xi-\gamma\big)}|x-z|^p+C_\gamma 
  \end{align*}
  Choose $\gamma>0$ sufficiently small such that simultaneously $\Xi-\gamma>0$ and 
  \begin{equation*}
    \rho\define \big(1-\a_{\eta,\bar{R}}\big)e^{p(\Lambda+\gamma)}+\a_{\eta,\bar{R}}e^{-p(\Xi-\gamma)}<1.
  \end{equation*}
  This shows that
  \begin{equation}\label{eq:estimate_p_bigger}
    \Expec{\big|R_{1}-S_{1}\big|^p}\leq\rho|x-y|^p+C_\gamma.
  \end{equation}
  It is clear that the estimate \eqref{eq:estimate_p_bigger} also holds for $p<2$ with the constant
  \begin{equation*}
    C_\gamma=\frac{1}{(2\gamma)^{\frac{p}{2}}}\left(\frac{\|\varsigma\|_{\Omega_\beta}}{t^{\beta}}+|\dot{\bar{\ell}}|_\infty\right)^p
  \end{equation*}
   and a slightly increased $\rho<1$. Since $P^t=\bar{B}_{t+\cdot}$, \cref{lem:smooth_part_decay} and the identity \eqref{eq:expec_diff} show that
  \begin{equation*}
    \Expec{\big|X_{t+1}-Y_{t+1}\big|^p}\leq\rho\Expec{\big|X_{t}-Y_{t}\big|^p}+\frac{C\big(1+\|\varsigma\|_{\Omega_\beta}^p\big)}{t^{p\beta}}
  \end{equation*}
  for some numerical constant $C>0$ independent of $t$ and $\varsigma$. Therefore, iterating this bound we find
  \begin{equation}\label{eq:quenched_wasserstein}
    \Expec{\big|X_{t}-Y_{t}\big|^p}\lesssim e^{-ct}\Expec{|X_0-Y_0|^p}+C\big(1+\|\varsigma\|_{\Omega_\beta}^p\big)\sum_{i=0}^{[t]-2}\frac{\rho^i}{(t-1-i)^{p\beta}}.
  \end{equation}
  The last sum is easily seen to be $\lesssim t^{-p\beta}$ uniformly in $t\geq 2$ and the claim follows at once.
\end{proof}

By a strategy inspired by \cite[Section 7]{Panloup2020} (see also \cite{Hairer2005}), we can lift \cref{prop:conditional_stationary_wasserstein} to a total variation bound. Since the exposition of Panloup and Richard does not immediately transfer to the problem at hand, we choose to include the necessary details.  Consider the system
\begin{equation}\label{eq:girsanov_coupling}
  \begin{aligned}[c]
    dX_s&=b(X_s)\, ds +d\varsigma_s+\sigma d\tilde{B}_s,\\
    dZ_s&=b(Z_s)\, ds +\sigma\,dB_s+\sigma\varphi^t(s)\,ds,
    \end{aligned}
\end{equation}
where $X_0$ is an arbitrary initial condition and $Z$ is the stationary solution of the first equation. Our aim is to exhibit an adapted integrable function $\varphi^t:[0,t+1]\to\R^n$ which vanishes on $[0,t]$ and ensures that $X_{t+1}=Z_{t+1}$. To this end, we define
\begin{equation}\label{eq:coupling_function}
  \varphi^t(s)\define\left\{
  \begin{array}{ll}
  \left(2\frac{|X_t-Z_t|^{\frac12}}{|X_s-Z_s|^{\frac12}}+\lambda\right)\sigma^{-1}(X_s-Z_s)-\dot{\bar{B}}_s+\sigma^{-1}\dot{\varsigma}_s, \quad \quad &s\in [t,t+1],\\
  0, \qquad \qquad &\hbox{ otherwise.}
  \end{array}\right.
\end{equation}
\begin{lemma}\label{lem:girsanov}
 Let $t\geq 1$,   $\varsigma\in\Omega_\alpha$, $b\in\S(\kappa,R,\lambda)$, 
   and consider the system \eqref{eq:girsanov_coupling}  with $\varphi^t$ defined in \eqref{eq:coupling_function}. Then $X_{t+1}=Z_{t+1}$ and, for any $\beta<\alpha\wedge(1-H)$,
\begin{equation}\label{eq:phi_norm}
  |\varphi^t|_\infty\lesssim |X_t-Z_t|+\frac{\|\varsigma\|_{\Omega_\beta}+\|\bar{B}\|_{\Omega_\beta}}{t^{\beta}},\qquad |\dot{\varphi}^t|_\infty\lesssim |X_t-Z_t|^{\frac12}+|X_t-Z_t|+\frac{\|\varsigma\|_{\Omega_\beta}+\|\bar{B}\|_{\Omega_\beta}}{t^{1+\beta}},
\end{equation}
where the derivative of $\varphi^t$ is understood as right- and left-sided derivative at the boundaries $t$ and $t+1$, respectively.
\end{lemma}
\begin{proof}
  The argument is a minor modification of \cite[Lemma 5.8]{Hairer2005}: Abbreviate $f(s)\define|X_s-Z_s|^2$, then
  \begin{equation*}
    f^\prime(s)=2\braket{b(X_s)-b(Z_s)+\dot{\varsigma}_s-\sigma\dot{\bar{B}}_s-\sigma\varphi^t(s),X_s-Z_s}\leq -4|X_t-Z_t|^{\frac12}f(s)^{\frac34}
  \end{equation*}
  since $b\in\S(\kappa,R,\lambda)$. It follows that
  \begin{equation*}
    |X_s-Z_s|^{\frac12}\leq |X_t-Z_t|^{\frac12}-(s-t)|X_t-Z_t|^{\frac12}\qquad\forall\, s\in[t,t+1],
  \end{equation*}
  whence $X_{t+1}=Z_{t+1}$. This also implies
  \begin{equation*}
    \left|\frac{d}{ds}\big(X_s-Z_s\big)\right|\leq\big(\Lip{b}+2+\lambda\big)|X_t-Z_t|^{\frac12}|X_s-Z_s|^{\frac12}
  \end{equation*}
  and consequently
  \begin{equation*}
    \left|\frac{d}{ds}\left(\frac{X_s-Z_s}{|X_s-Z_s|^{\frac12}}\right)\right|\leq\frac32\frac{\left|\frac{d}{ds}\big(X_s-Z_s\big)\right|}{|X_s-Z_s|^{\frac12}}\lesssim|X_t-Z_t|^{\frac12}.
  \end{equation*}
  The bounds \eqref{eq:phi_norm} follow at once.
\end{proof}
\begin{remark}
  We stress that the bound on $|\dot{\varphi}^t|_\infty$ only holds for a Lipschitz continuous drift $b$.
\end{remark}

It is now easy to prove the following result:
\begin{proposition}\label{prop:conditional_stationary_total}
  Assume the conditions of \cref{prop:conditional_stationary_wasserstein} for $p=1$. Then, for any $\beta<\alpha\wedge(1-H)$, it holds that
  \begin{equation*}
    \TV{\L\big(\Psi^\vsigma_t(Y)\big)-\pi}\lesssim {t^{-\frac{\beta}{3}}\big(1+\|\varsigma\|_{\Omega_\beta}\big)\big(1+\W^1(\L(Y),\pi)\big)}  \qquad \forall\, t>0.
  \end{equation*}
\end{proposition}

\begin{proof}
  Let $B$ and $B^\prime$ be $H$-fBms built from underlying two-sided Wiener processes $W$ and $W^\prime$, see \eqref{eq:mandelbrot}. Recall that $\tilde{B}$ is the Riemann-Liouville process associated with $B$. Let $X$ and $Z$ solve 
  \begin{align}\label{eq:proof_tv_equations}
    \begin{split}
    dX_s&=b(X_s)\,ds+d\varsigma_s+\sigma d\tilde{B}_s,\\
    dZ_s&=b(Z_s)\,ds+\sigma\,dB^\prime_s,
    \end{split}
  \end{align}
  where $X_0\overset{d}{=}Y$ and $Z$ is the stationary solution. 
    Fix $t>1$.
    We shall use the bound
  \begin{align}
    \TV{\L\big(\Psi^\vsigma_{t+1}(Y)\big)-\pi}&=\inf_{(\tilde B,B^\prime)}\prob\big(X_{t+1}\neq Z_{t+1}\big)\leq\inf_{(W,W^\prime)}\prob\big(X_{t+1}\neq Z_{t+1}\big) 
    \nonumber\\
    &\leq\inf_{(W,W^\prime)}\prob\big(X_{t+1}\neq Z_{t+1},|X_t-Z_t|\leq\delta\big)+ \inf_{(W,W^\prime)}\prob\big(|X_t-Z_t|>\delta\big).\label{eq:tv_proof}
  \end{align}
 Taking $W$ and $W'$ equal, we are in the setting of \cref{prop:conditional_stationary_wasserstein}. The estimate \eqref{eq:quenched_wasserstein} thus shows that, for any $\delta\in(0,1]$,
\begin{equation}
\label{eq:tv_proof_2}
\inf_{(W,W^\prime)}\prob\big(|X_t-Z_t|>\delta)\nonumber \le \frac{C\big(1+\|\varsigma\|_{\Omega_\beta}\big)\big(1+\W^1(\L(Y),\pi)\big)}{\delta t^\beta}.
\end{equation}
 To bound the first term in \eqref{eq:tv_proof} we exploit the fact that $X_t$ and $Z_t$ are already close so that we can couple them at time $t+1$ with a controlled cost. 
 Let $\varphi^t$ be the function from \cref{lem:girsanov}; in particular $\varphi^t(s)=0$  for $s<t$.
 We observe that $B^\prime=B+\int_0^\cdot\varphi^t(s)\,ds$ on $[0,t+1]$ if and only if $W^\prime=W+\int_{-\infty}^\cdot\psi^t(s)\,ds$ on $(-\infty,t+1]$, where for a suitable constant $\gamma_H\in\R$,
  \begin{equation*}
    \psi^t(s)=\begin{cases}\displaystyle\gamma_H\frac{d}{ds}\int_{t}^s (s-u)^{\frac12-H}\varphi^t(u)\,du, & s\in[t,t+1], \\
    0, & s\in(-\infty,t),
    \end{cases}
  \end{equation*}
  see \cite[Lemma 4.2]{Hairer2005} for details. Let $\T$ be the linear transformation $\T(w)=w+\int_{-\infty}^\cdot\psi^t(s)\,ds$, $w\in\C\big((-\infty,t+1],\R^n\big)$. We return to the equations \eqref{eq:proof_tv_equations}. By the construction, we have $X_{t+1}=Z_{t+1}$ for those realizations for which $W^\prime=\T(W)$.  In particular,
\begin{equation*}
  \inf_{(W,W^\prime)}\prob\big(X_{t+1}\neq Z_{t+1},|X_t-Z_t|\leq \delta\big) \leq \inf_{(W, W')} \prob\big(W' \not = \T(W), |X_t-Z_t|\leq \delta\big).
\end{equation*}
Throughout this proof, all stochastic processes are considered only up to time $t+1$. Let $\hat \prob(\cdot)=\prob\big(\cdot\,\big|\, |X_t-Z_t|\leq \delta\big)$. Let $\mu$ and $\nu$ denote the laws under $\hat\prob$ of $W$ and $\T(W)$, respectively.
We take an optimal coupling $(\hat W, \hat W^\prime)$ achieving the total variation distance $ \TV{\mu-\nu}$. Then the Pinsker-Csizsar inequality, a consequence of the Girsanov theorem, shows that
 \begin{align*}
    \inf_{(W,W^\prime)}\prob\big(X_{t+1}\neq Z_{t+1},|X_t-Z_t|\leq\delta\big)
    &\leq \TV{\mu-\nu}\prob\big(|X_t-Z_t|\leq \delta \big)\\
 &\leq\frac{1}{2} \E_{\hat \prob}  \left[ \int_t^{t+1}|\psi^t(s)|^2\,ds\right]^{\frac12}\prob\big(|X_t-Z_t|\leq\delta\big).   
 \end{align*}
On integration by parts we find
  \begin{equation*}
    \int_t^{t+1}|\psi^t(s)|^2\,ds\lesssim\begin{cases}
      |\varphi^t|_\infty^2,& H<\frac12,\\
      |\varphi^t|^2_\infty+|\dot{\varphi}^t|_\infty^2,& H>\frac12.
    \end{cases}
  \end{equation*}
 In either case, \eqref{eq:phi_norm} yields
  \begin{equation*}
    \int_t^{t+1}|\psi^t(s)|^2\,ds\lesssim \delta+\frac{\|\varsigma\|^2_{\Omega_\beta}+\|\bar{B}\|^2_{\Omega_\beta}}{t^{2\beta}}
  \end{equation*}
  on the event $\{|X_t-Z_t|\leq \delta\}$ and therefore
  \begin{align*}
    &\phantom{\leq}\inf_{(W,W^\prime)}\prob\big(X_{t+1}\neq Z_{t+1},|X_t-Z_t|\leq\delta\big) 
    \lesssim\sqrt{\delta}+\frac{\|\varsigma\|_{\Omega_\beta}+\big\|\|\bar{B}\|_{\Omega_\beta}\big\|_{L^2}}{t^\beta}.
  \end{align*}
Combining this with \eqref{eq:tv_proof} and \cref{lem:smooth_part_decay}, we have proven
  \begin{equation}\label{eq:tv_final}
    \TV{\L\big(\Psi_{t+1}^\varsigma(Y)\big)-\pi}\lesssim \big(1+\|\varsigma\|_{\Omega_\beta}\big)\big(1+\W^1(\L(Y),\pi)\big)\left(\sqrt{\delta}+\frac{1}{\delta t^\beta}\right),
  \end{equation}
  which is minimized for $\delta=t^{-\frac{2\beta}{3}}$.
\end{proof}

By duality and \cref{lem:conditioning}, we obtain the following ergodic theorem as a corollary to \cref{prop:conditional_stationary_wasserstein,prop:conditional_stationary_total}. It provides the fundamental estimates for our proof of the averaging principle for the fractional slow-fast system with feedback dynamics.
\begin{corollary}\label{cor:total_variation_conditional}
  Let $0\leq s\leq t$ and let $X,Y$ be $\F_s$-measurable random variables. Suppose that there are $\kappa,R>0$ such that $b(x,\cdot)\in\S_1(\kappa,R)$ for every $x\in\R^d$. Then, for any $\zeta<1-\hat{H}$ and
  \begin{enumerate}
    \item\label{it:ergodicity_wasserstein} any Lipschitz function $h:\R^d\times\R^n\to\R$, 
    \begin{equation*}
    \Big|\Expec{h\big(X,\bar{\Phi}_{s,t}^X(Y)\big)-\bar{h}(X)\,|\,\F_s}\Big|\lesssim\Lip{h}\Big(1+\big\|\varepsilon^{-\hat{H}}\bar{\hat{B}}_{\varepsilon\cdot}^s\big\|_{\Omega_\zeta}\Big)\big(1+|Y|\big)\left(1\wedge\frac{\varepsilon^{\zeta}}{|t-s|^{\zeta}}\right).
  \end{equation*}
  \item\label{it:ergodicity_tv} any bounded measurable function $h:\R^d\times\R^n\to\R$,
  \begin{equation*}
    \Big|\Expec{h\big(X,\bar{\Phi}_{s,t}^X(Y)\big)-\bar{h}(X)\,|\,\F_s}\Big|\lesssim |h|_\infty\Big(1+\big\|\varepsilon^{-\hat{H}}\bar{\hat{B}}_{\varepsilon\cdot}^s\big\|_{\Omega_\zeta}\Big)\big(1+|Y|\big)\left(1\wedge\frac{\varepsilon^{\frac{\zeta}{3}}}{|t-s|^{\frac{\zeta}{3}}}\right).
  \end{equation*}
  \end{enumerate}
  Here, as usual, $\bar{h}(x)=\int_{\R^n} h(x,y)\,\pi^x(dy)$.
\end{corollary}

\subsection{Geometric Ergodicity for SDEs Driven by Fractional Brownian Motion}\label{sec:geometric_ergodicity}

Applying the arguments of \cref{prop:conditional_initial_condition_wasserstein,prop:conditional_stationary_total} to the equation \begin{equation}\label{eq:sde}
  dY_t=b(Y_t)\,dt+\sigma\,dB_t,   
\end{equation}   
we obtain an exponential rate of convergence improving the known results:
\begin{proof}[Proof of \cref{thm:geometric}]
 In \cref{conveergence-equilibrium} we have already proven the  Wasserstein decay \eqref{eq:wasserstein_time_t}:
    \begin{equation*}
        \W^p(\mathcal{L}(Y_t),\pi)\leq Ce^{-ct}\bW^p\big(\mu,\pi\big), \qquad \forall\, t\ge 0
    \end{equation*}
    The total variation rate \eqref{eq:tv_process} then follows by a similar Girsanov coupling as in the proof of \cref{prop:conditional_stationary_total}. In fact, we now consider
  \begin{align*}
    dX_s&=b(X_s)\,ds+\sigma dB_s,\\
    dZ_s&=b(Z_s)\,ds+\sigma\,dB_s +\sigma\varphi^t(s)\,ds,
  \end{align*}
  where $X$ is started in the generalized initial condition $\mu$ and $Z$ is the stationary solution. Let us define
  \begin{equation*}
    \varphi^t(s)\define -\left(\frac{4|X_t-Z_t|^{\frac12}}{|X_s-Z_s|^{\frac12}}+\lambda\right)\sigma^{-1}(X_s-Z_s)\1_{[t,t+1]}(s).
  \end{equation*}
  It can then be checked similarly to \cref{lem:girsanov} that $X_{t+1}=Y_{t+1}$ and
  \begin{equation*}
    |\varphi^t|_\infty\lesssim |X_t-Z_t|,\qquad |\dot{\varphi}^t|_\infty\lesssim |X_t-Z_t|^{\frac12}+|X_t-Z_t|.
  \end{equation*}
  Consequently, the estimate \eqref{eq:tv_final} becomes
  \begin{equation}\label{eq:geometric_time_t}
    \TV{\L(Y_{t+1})-\pi}\lesssim\bW^1\big(\mu,\pi\big)\left(\sqrt{\delta}+\frac{e^{-ct}}{\delta}\right)
  \end{equation}
  and choosing $\delta= e^{-\frac{ct}{2}}$ shows a geometric decay of the total variation distance at a fixed time. To get asserted decay on the path space \eqref{eq:tv_process}, we observe that, by the very same argument as in \cite[Proposition 7.2 (iii)]{Panloup2020}, $\varphi^t$ actually induces a coupling on the path space with a similar cost. Hence, $\TV{\L(Y_{t+\cdot})-\prob_\pi}$ is still bounded by a quantity proportional to the right-hand side of \eqref{eq:geometric_time_t} and \eqref{eq:tv_process} follows at once. 
\end{proof}
\begin{remark}\label{rem:constant_xi}
    The admissible repulsivity strength $\Lambda(\kappa,R,p)$ obtained in the proof of \cref{thm:geometric} is certainly not optimal. We therefore abstain from deriving a quantitative upper bound. Let us however indicate one way to obtain such an estimate: Start from \eqref{eq:quant_lower_bound} in the proof \cref{lem:probabilistic_control} and recall a standard result (see e.g. \cite[Theorem D.4]{Piterbarg2012}) saying that
    \begin{equation*}
      \prob\big(|\tilde{B}|_\infty\leq|\sigma|^{-1}\varepsilon\big)\geq 1-K\big(|\sigma|^{-1}\varepsilon\big)^{\frac{1}{H}}e^{-H(|\sigma|^{-1}\varepsilon)^2}
    \end{equation*}
    for a known numerical constant $K>0$. Finally optimize over all constants involved.
\end{remark}

Let us finally sketch the main differences for a more general Gaussian driving noise $G$ in equation \eqref{eq:sde}. We assume that $G$ has continuous sample paths and a moving average representation similar to \eqref{eq:mandelbrot} with a kernel $\mathfrak{G}:\R\to\Lin{n}$ which vanishes on $(-\infty,0]$, is continuous on $(0,\infty)$, and satisfies
\begin{equation*}
	\int_{-\infty}^t\big|\mathfrak{G}(t-u)-\mathfrak{G}(-u)\big|^2\,du<\infty
\end{equation*}
for each $t>0$. Then
\begin{equation*}
	G_t=\int_{-\infty}^t \mathfrak{G}(t-u)-\mathfrak{G}(-u)\,dW_u,\qquad t\geq 0,
\end{equation*}
has the locally independent increment decomposition 
\begin{equation*}
	\big(\theta_t G\big)_h=\int_{-\infty}^t\mathfrak{G}(t+h-u)-\mathfrak{G}(t-u)\,dW_u+\int_t^{t+h}\mathfrak{G}(t+h-u)\,dW_u\define\bar{G}^t_h+\tilde{G}^t_h
\end{equation*} 
with respect to any compatible filtration. Moreover, we require that
\begin{equation*}
  \int_0^{\delta} |\mathfrak{G}(u)|\,du>0
\end{equation*}
for each $\delta>0$. We remark that (up to a time-shift) this is certainly implied by the assumptions of Panloup and Richard, see \cite[Condition $\boldsymbol{(\mathrm{C}_2)}$]{Panloup2020}. As we have seen in \cref{ex:titmarsh}, the Cameron-Martin space of $(\tilde{G}_h)_{h\in[0,1]}$ then densely embeds into $\C_0([0,1],\R^n)$. Thus \cref{rem:Wasserstein-general} applies and we obtain a geometric rate in Wasserstein distance, provided that there is a stationary measure for the equation $dY_t=b(Y_t)\,dt+\sigma\,dG_t$.

\section{The Fractional Averaging Principle}\label{sec:feedback}

Let us remind the reader of the setup of \cref{thm:feedback_fractional}: We consider the slow-fast system
\begin{alignat}{4}
  dX_t^\varepsilon&=f(X_t^\varepsilon,Y_t^\varepsilon)\,dt+g(X_t^\varepsilon,Y_t^\varepsilon)\,dB_t,& \qquad&X_0^\varepsilon=X_0,\label{eq:slow_feedback_sec}\\
  dY_t^\varepsilon&=\frac{1}{\varepsilon}b(X_t^\varepsilon,Y_t^\varepsilon)\,dt+\frac{1}{\varepsilon^{\hat{H}}}\sigma\,d\hat{B}_t,&\qquad &Y_0^\varepsilon=Y_0,\label{eq:fast_feedback_sec}
\end{alignat}
driven by independent $d$-dimensional and $n$-dimensional fractional Brownian motions $B$ and $\hat{B}$ with Hurst parameters $H\in(\frac12,1)$ and $\hat{H}\in(1-H,1)$, respectively. We claim that $X_t^\epsilon$ converges to the solution of the na\"ively averaged equation \eqref{eq:effective_dynamics} as $\varepsilon\to 0$. 

Let us also introduce the following filtrations for later reference:
\begin{equation*}
  \G_t\define\sigma(B_s,s\leq t),\quad\hat{\G}_t\define\sigma(\hat{B}_s,s\leq t),\quad \F_t\define\G_t\vee\hat{\G}_t.
\end{equation*}
To be utterly precise, we actually use the right-continuous completion of $\F$ in order to ensure that hitting time of an open sets by a continuous, adapted process is a stopping time. Observe that $\F$ is compatible with the fBm $\hat{B}$, see \cref{sec:increment}.

We shall first convince ourselves that, under the conditions of \cref{thm:feedback_fractional}, the pathwise solution of the  slow-fast system \eqref{eq:slow_feedback_sec}--\eqref{eq:fast_feedback_sec} exists globally. If the drift vector field $b:\R^d\times\R^n\to\R^n$ in \eqref{eq:fast_feedback_sec} were \emph{globally} Lipschitz continuous, this would be an easy consequence of the standard Young bound \cite{Young1936}:
\begin{equation}\label{eq:young}
  \left|\int_s^t f_r\,d\h_r\right|\lesssim |f|_{\C^\beta}|\h|_{\C^\alpha}|t-s|^{\alpha+\beta}+|f_s||\h|_{\C^\alpha}|t-s|^\alpha,
\end{equation}
provided that $\alpha+\beta>1$. We shall also prove a bound on the moments of the H\"older norm of the solution for any fixed scale $\epsilon$. The main technical estimates in the proof of \cref{thm:feedback_fractional} are delegated to \cref{sec:uniform_bounds}, allowing us to easily conclude the argument in \cref{sec:proof} by appealing to L\^e's stochastic sewing lemma \cite{Le2020}. 

\subsection{A Solution Theory for the Slow-Fast System}\label{sec:solution_theory}

We shall begin with a deterministic (pathwise) existence and uniqueness result. Fix a terminal time $T>0$ and let $\h=(\h^1,\h^2)\in\C^{\alpha_1}([0,T],\R^{m})\times\C^{\alpha_2}([0,T],\R^{n})$, where $\alpha_1>\frac12$ and $\alpha_2>1-\alpha_1$. We consider the Young differential equation
\begin{equation}\label{eq:ode}
  z(t)=\begin{pmatrix}z^1(t)\\z^2(t)\end{pmatrix}=z_0+\int_0^t \begin{pmatrix}F_1\big(z(s)\big)\\F_2\big(z(s)\big)\end{pmatrix}\,ds+\int_0^t G\big(z(s)\big)\,d\h_s.
\end{equation}
We impose the following assumptions on the data:

\begin{condition}\label{cond:data_ode}
\leavevmode
\begin{enumerate}
  \item\label{it:cond_ode_1} $F_1:\R^{d}\times\R^{n}\to\R^{d}$ is bounded and globally Lipschitz continuous.
  \item\label{it:cond_ode_2} $F_2:\R^{d}\times\R^{n}\to\R^{n}$ is locally Lipschitz continuous and of linear growth, that is, $|F_2(z,x)|\lesssim 1+|x|+|z|$ for all $x\in\R^n$ and $z\in\R^d$. Moreover, there are $\kappa,D>0$ such that
  \begin{equation*}
  \Braket{F_2(z, x)-F_2(z,y),x-y}\leq D- \kappa|x-y|^2 \qquad \forall\,   x,y\in\R^n, \forall \,z\in \R^d.
\end{equation*}
  \item\label{it:cond_ode_3} $G:\R^{d}\times\R^{n}\to\Lin[m+n]{d+n}$ is of the form $G=\begin{pmatrix}G_1 & 0\\ 0 & G_2\end{pmatrix}$ with $G_1\in\Cb{2}\big(\R^{d}\times\R^{n},\Lin[m]{d}\big)$ and $G_2\in\Lin{d}$ is constant.
\end{enumerate}
\end{condition}

Our proof for the well-posedness of \eqref{eq:ode} and the non-explosiveness is based on the following comparison lemma, versions of which will be of repeated use in the sequel:
\begin{lemma}\label{lem:comparison}
  Let $F_2:\R^{d}\times\R^{n}\to\R^{n}$ satisfy \cref{cond:data_ode} \ref{it:cond_ode_2} and let $\varsigma\in\C_0(\R_+,\R^{n})$, $z\in\C(\R_+,\R^{d})$. 
  \begin{enumerate}
    \item Then for any $x_0\in\R^{n}$, there are unique global solutions to
    \begin{equation*}
        x(t)=x_0+\int_0^t F_2\big(z(s),x(s)\big)\,ds+\varsigma_t,\qquad y(t)=x_0-\int_0^ty(s)\,ds+\varsigma_t.
      \end{equation*}
      Furthermore, on any finite time interval $[0,T]$, the difference of the solutions satisfies the bound
      \begin{equation}\label{eq:solution_difference}
        |x(t)-y(t)|^2\lesssim\int_0^t e^{-\kappa(t-s)}\big(1+|y(s)|+|z(s)|\big)^2\,ds
      \end{equation}
      for all $t\in[0,T]$. In particular,
      \begin{equation}\label{eq:a_priori_sup}
        |x|_\infty\lesssim 1+|x_0|+|\varsigma|_\infty+|z|_\infty.
      \end{equation}

      \item  If, in addition, $\varsigma\in\C^\alpha([0,T],\R^{n})$ for some $\alpha>0$, then $x\in\C^\alpha([0,T],\R^{n})$ and the following bound holds:
      \begin{equation}\label{eq:comparison_apriori} 
        |x|_{\C^{\alpha}}\lesssim 1+|x_0|+|z|_\infty+|\varsigma|_{\C^{\alpha}}.
      \end{equation}
  \end{enumerate}
\end{lemma}
\begin{proof}    
  Since $F_2$ is locally Lipschitz, it is clear that uniqueness holds for the equation defining $x$. To see existence, first notice that $\tilde{x}(t)\define x(t)-\varsigma_t$ solves
  \begin{equation*}
    \tilde{x}(t)=x_0+\int_0^t F_2\big(z(s),\tilde{x}(s)+\varsigma_s\big)\,ds.
  \end{equation*}
  Set $\Upsilon(s,x) =F_2\big(z(s), x+\varsigma_s\big)$. This function is jointly continuous in $(s,x)$. Therefore, a local solution exists by the Carath\'eodory theorem. 

  On the other hand, global existence and uniqueness of $y$ is standard. Consequently, the required non-explosion statement follows easily upon establishing \eqref{eq:solution_difference}. To this end, we first observe that, for all $z\in\R^{d}$ and all $x,y\in\R^{n}$, the off-diagonal large scale contraction property and the linear growth of $F$ furnish the following bound:
  \begin{align*}
    \Braket{F_2(z,x)+y,x-y}&\leq D-\kappa|x-y|^2+\<F_2(z,y)+y,x-y\>\\
    &\leq D- \frac{\kappa}{2}|x-y|^2+\frac{C}{\kappa}\big(1+|z|+|y|\big)^2
  \end{align*}
  for some uniform constant $C>0$, where we also used Young's inequality. Consequently, the function $h(t)\define e^{\kappa t}|x(t)-y(t)|^2$ satisfies 
  \begin{equation*}
    h^\prime(t)\lesssim e^{\kappa t}\big(1+|y(t)|+|z(t)|\big)^2
  \end{equation*}
  and \eqref{eq:solution_difference} follows at once. 

  The bound \eqref{eq:comparison_apriori} is an immediate consequence of \eqref{eq:solution_difference} together with the fact that
    \begin{equation*}
    |x|_{\C^{\alpha}}\lesssim|F_2(z,x)|_{\infty}T^{1- \alpha}+|\varsigma|_{\C^{\alpha}}\lesssim \big(1+|z|_\infty+|x|_{\infty}\big) T^{1-\alpha}+|\varsigma|_{\C^{\alpha}}.\qedhere
  \end{equation*}
\end{proof}

The announced existence and uniqueness result for \eqref{eq:ode} is as follows:
\begin{proposition}\label{prop:abstract_ode}
  Under \cref{cond:data_ode}, for any $T>0$ and any $\beta<\alpha_1\wedge\alpha_2$, \eqref{eq:ode} has a unique global solution in $\C^{\beta}([0,T],\R^{d+n})$. 
\end{proposition}
\begin{proof}
Owing to \cref{lem:comparison}, it is enough to derive an \emph{a priori} bound on $|z^1|_{\C^{\tilde{\alpha}}}$, $\tilde{\alpha}\in[\beta,\alpha_1)$, to conclude with a standard Picard argument. 

Let $\delta\in(0,1)$. By the Young bound \eqref{eq:young}, we see that
  \begin{align*}
    |z^1|_{\C^{\tilde{\alpha}}}&\lesssim |F_1|_\infty\delta^{1-\tilde{\alpha}}+\big(\big|G_1(z^1,z^2)\big|_{\C^{\tilde{\alpha}\wedge\alpha_2}}+|G_1|_\infty\big)|\h^1|_{\C^{\tilde{\alpha}}}\\
    &\lesssim\big(1+|z^1|_{\C^{\tilde{\alpha}}}+|z^2|_{\C^{\alpha_2}}\big)\big(1+|\h^1|_{\C^{\alpha_1}}\big)\delta^{\alpha_1-\tilde{\alpha}},
  \end{align*}
  where the prefactor is proportional to $M\define|F_1|_\infty+|G|_\infty+\Lip{G}$. We may apply \cref{lem:comparison} to $z^2$ to further find
 \begin{equation*}
    |z^1|_{\C^{\tilde{\alpha}}}\lesssim
    \big(1+|z^1|_{\C^{\tilde{\alpha}}}+|z_0|+|\h^2|_{\C^{\alpha_2}}\big)\big(1+|\h^1|_{\C^{\alpha_1}}\big)\delta^{\alpha_1-\tilde{\alpha}}.
  \end{equation*}
  Here, we take the H\"older norms of $z^1,z^2$ over the interval $[0,\delta]$, whereas we use the full interval $[0,T]$ for $\h^1$ and $\h^2$. For $\delta>0$ small enough, we therefore get
  \begin{equation}\label{eq:iteration}
    |z^1|_{\C^{\tilde{\alpha}}([0,\delta])}\lesssim\big(1+|z_0|+|\h^2|_{\C^{\alpha_2}}\big)\big(1+|\h^1|_{\C^{\alpha_1}}\big).
  \end{equation}
  Combining this with \cref{lem:comparison}, we can find a constant $C>0$ such that
  \begin{equation*}
    |z(\delta)|\leq|z_0|+|z^1|_{\C^{\tilde{\alpha}}([0,\delta])}+|z^2|_{\C^{\alpha_2}([0,\delta])}\leq C\big(1+|z_0|+|\h^2|_{\C^{\alpha_2}}\big)\big(1+|\h^1|_{\C^{\alpha_1}}\big).
  \end{equation*}
  This bound can now be easily iterated and together with \eqref{eq:iteration} we see that there is a (increased) constant $C$ such that
  \begin{equation*}
    |z^1|_{\C^{\tilde{\alpha}}([t,t+\delta])}\lesssim\big(1+|z_t|+|\h^2|_{\C^{\alpha_2}}\big)\big(1+|\h^1|_{\C^{\alpha_1}}\big)\leq C^{\left[\frac{t}{\delta}\right]+1}\big(1+|z_0|+|\h^2|_{\C^{\alpha_2}}\big)\big(1+|\h^1|_{\C^{\alpha_1}}\big)^{\left[\frac{t}{\delta}\right]+2}
  \end{equation*}
  for each $t\in[0,T-\delta]$. Since $|\cdot|_{\C^{\tilde{\alpha}}([0,T])}\leq 2\delta^{\tilde{\alpha}-1}\sup_t |\cdot|_{\C^{\tilde{\alpha}}([t,t+\delta])}$, we get that
  \begin{equation}\label{eq:a_priori}
    |z^1|_{\C^{\tilde{\alpha}}([0,T])}\leq\frac{2C^{\left[\frac{t}{\delta}\right]+1}}{\delta^{1-\tilde\alpha}}\big(1+|z_0|+|\h^2|_{\C^{\alpha_2}}\big)\big(1+|\h^1|_{\C^{\alpha_1}}\big)^{\left[\frac{T}{\delta}\right]+2}.
  \end{equation}

  Local existence and uniqueness of a solution to \eqref{eq:ode} is a classical consequence of the Young bound. Indeed, if we define 
  \begin{equation*}
    A_\delta\define\left\{f\in\C^{\beta}([0,\delta],\R^{d+n}):\,f(0)=z_0\text{ and }|f|_{\C^{\beta}}\leq 1\right\},
  \end{equation*}
  then, for $\delta>0$ small enough, the operator $\mathcal{A}_\delta: A_\delta\to A_\delta$,
  \begin{equation*}
    (\mathcal{A}_\delta z)(t)\define z_0+\int_0^t\begin{pmatrix}F_1\big(z(s)\big)\\F_2\big(z(s)\big)\end{pmatrix}\,ds+\int_0^t G\big(z(s)\big)\,d\h_s,
  \end{equation*}
   is contracting on a complete metric space. Abbreviating $\gamma\define\alpha_1\wedge\alpha_2$, this in turn follows from the well-known bounds
  \begin{align*}
    \left|\int_0^\cdot G\big(z(s)\big)\,d\h_s\right|_{\C^{\beta}}&\lesssim(\Lip{G}+|G|_\infty)(|z|_{\C^{\beta}}+1)|\h|_{\C^{\gamma}}\delta^{\gamma-\beta},\\
    \left|\int_0^\cdot G\big(z(s)\big)-G\big(\bar{z}(s)\big)\,d\h_s\right|_{\C^{\beta}}&\lesssim (\Lip{G}+\Lip{DG})|\h|_{\C^{\gamma}}\delta^{\gamma-\beta}|z-\bar{z}|_{\C^{\beta}},\\
    \left|\int_0^\cdot \begin{pmatrix}F_1\big(z(s)\big)\\F_2\big(z(s)\big)\end{pmatrix}\,ds\right|_{\C^{\beta}}&\leq \big(|F_1|_{\infty;\,B_{\delta^{\beta}}(z_0)}+|F_2|_{\infty;\,B_{\delta^{\beta}}(z_0)}\big)\delta^{1-\beta},\\
    \left|\int_0^\cdot \begin{pmatrix}F_1\big(z(s)\big)-F_1\big(\bar{z}(s)\big)\\F_2\big(z(s)\big)-F_2\big(\bar{z}(s)\big)\end{pmatrix}\,ds\right|_{\C^{\beta}}&\leq \big(\Lip{F_1}+\Lip[B_{\delta^{\beta}}(z_0)]{F_2}\big)\delta|z-\bar{z}|_{\C^{\beta}}
  \end{align*}
  for all $z,\bar{z}\in A_\delta$, where $|\cdot|_{\infty;\,A}$ and $\Lip[A]{\cdot}$ denote the respective norms of the function restricted to the set $A$. Here, we also used that $\max\big(|z-z_0|_\infty,|\bar z-z_0|_\infty\big)\leq\delta^\beta$ since $z,\bar{z}\in A_\delta$ by assumption. Consequently, there is a unique solution to \eqref{eq:ode} in $\C^{\beta}([0,\delta],\R^{d+n})$. Global existence and uniqueness follow from the \emph{a priori} estimates \eqref{eq:comparison_apriori} and \eqref{eq:a_priori} by a standard maximality argument.
\end{proof}

We now bring the randomness back in the picture. To this end, let $\alpha>0$, $p\geq 1$, and $T>0$. We define the space
\begin{equation*}
   \B_{\alpha,p}([0,T],\R^d)\define\left\{X:[0,T]\times\Omega\to\R^d:\,X\text{ is }(\F_t)_{t\in[0,T]}\text{-adapted and }\|X\|_{\B_{\alpha,p}([0,T],\R^d)}<\infty\right\},
\end{equation*} 
where we introduced the semi-norm
\begin{equation*}
  \|X\|_{\B_{\alpha,p}([0,T],\R^d)}\define\sup_{s\neq t\in[0,T]}\frac{\|X_t-X_s\|_{L^p}}{|t-s|^\alpha}.
\end{equation*}
If the terminal time $T$ and the dimension $d$ are clear from the context, we shall also write $\B_{\alpha,p}$ for brevity. By Kolmogorov's continuity theorem, we have the continuous embeddings
\begin{equation}\label{eq:embeddings}
  L^p\big(\Omega,\C^{\alpha+\delta}([0,T],\R^d)\big)\hookrightarrow\B_{\alpha,p}([0,T],\R^d)\hookrightarrow  L^p\big(\Omega,\C^{\alpha-\delta-\frac1p}([0,T],\R^d)\big)
\end{equation}
for any $\delta>0$. Finally, let us also introduce the Besov-type space
\begin{align*}
  W_0^{\alpha,\infty}([0,T],\R^d)&\define \big\{f:[0,T]\to\R^d:\,|f|_{\alpha,\infty}<\infty\big\},\\
  |f|_{\alpha,\infty}&\define\sup_{t\in[0,T]}\left(|f(t)|+\int_0^t\frac{|f(t)-f(s)|}{|t-s|^{\alpha+1}}\,ds\right).
\end{align*}
Nualart and R\u{a}s\c{c}anu proved the following classical result:
\begin{proposition}[{\cite[Theorem 2.1.II]{Rascanu2002}}]\label{prop:nualart}
   Let $f:\R^d\times\R^n\to\R^d$ be bounded Lipschitz continuous and $g:\R^d\times\R^n\to\Lin[m]{d}$ be of class $\C_b^2$. Let $(Y_t)_{t\in[0,T]}$ be a stochastic process with sample paths in $\C^{\gamma}([0,T],\R^n)$ for some $\gamma>1-H$ and let $B$ be an fBm with Hurst parameter $H>\frac12$. Then there is a unique global solution to the equation
   \begin{equation*}
     X_t=X_0+\int_0^t f(X_s,Y_s)\,ds+\int_0^tg(X_s,Y_s)\,dB_s
   \end{equation*}
   and, provided that $X_0\in L^\infty$, we also have that
   \begin{equation*}
     |X|_{\alpha,\infty}\in\bigcap_{p\geq 1} L^p
   \end{equation*}
   for each $\alpha<\frac12\wedge\gamma$.
\end{proposition} 

\begin{corollary}\label{cor:norm_bound_solution}
  Fix the scale parameter $\varepsilon>0$ and a terminal time $T>0$. Let $\alpha<H\wedge\hat{H}$. There is a unique pathwise solution $(X^\varepsilon,Y^\varepsilon)\in\C^{\alpha}([0,T],\R^{d+n})$ to the slow-fast system \eqref{eq:slow_feedback_sec}--\eqref{eq:fast_feedback_sec}. Moreover, for any $p\geq 1$ and any $\beta<\frac12\wedge\hat{H}$, we have that
  \begin{equation}\label{eq:b_norm_bound}
   \|X^\varepsilon\|_{\B_{\beta,p}}<\infty.
  \end{equation}
\end{corollary}
\begin{proof}
  The first part is an immediate consequence of \cref{prop:abstract_ode}. We stress that the bound \eqref{eq:b_norm_bound} does not follow from our \emph{a priori} estimate \eqref{eq:a_priori} since, by Fernique's theorem, 
  \begin{equation*}
    \Expec{\exp\left(a|B|_{\C^\beta}^2\right)}<\infty
  \end{equation*} 
  if and only if $a>0$ is sufficiently small. Instead, we employ \cref{prop:nualart}: Since $Y^\varepsilon\in\C^{\hat H-}([0,T],\R^n)$ by \cref{lem:comparison}, we see that, for each $\alpha<\frac12\wedge\hat H$, $|X^\varepsilon|_{\alpha,\infty}\in\bigcap_{p\geq 1} L^p$,  It is clear that
\begin{equation*}
  W_0^{\alpha,\infty}([0,T],\R^d)\hookrightarrow \C^{\alpha-\delta}([0,T],\R^d)
\end{equation*}
for any $\delta>0$.  Combine this with the continuous embedding \eqref{eq:embeddings} to conclude \eqref{eq:b_norm_bound}.
\end{proof}

\begin{remark}
  We finally record that \cref{prop:abstract_ode,cor:norm_bound_solution} are the only places in the proof of \cref{thm:feedback_fractional} which require a linear growth of the drift $b$, see \cref{cond:feedback}. In fact, the remainder of the argument would still work, \emph{mutatis mutandis}, under the weaker assumption of a polynomially growing drift, i.e., $|b(x,y)|\lesssim 1+|x|^N+|y|^N$ for some $N\in\N$. It is however unclear whether the solution to \eqref{eq:slow_feedback_sec}--\eqref{eq:fast_feedback_sec} exists globally in this case.
\end{remark}

\subsection{Uniform Bounds on the Slow Motions}\label{sec:uniform_bounds}

Our strategy in proving \cref{thm:feedback_fractional} is as follows: The integrals in \eqref{eq:slow_feedback_sec} are approximated by suitable Riemann sums, on which we then aim to establish uniform bounds. These estimates translate into bounds on the integrals in view of L\^e's stochastic sewing lemma \cite{Le2020}.

Fix a terminal time $T>0$ and let $\mathcal{S}^p$ denote the set of adapted two-parameter processes on the simplex with finite $p^\text{th}$ moments; in symbols:
\begin{equation*}
  \mathcal{S}^p\define\left\{A:[0,T]^2\times\Omega\to\R^d:\,A_{s,t}=0\text{ for }s\geq t\text{ and }A_{s,t}\in L^p(\Omega,\F_t,\prob)\text{ for all }s,t\geq 0\right\}.
\end{equation*}
Given $\eta,\bar{\eta}>0$, we define the spaces
\begin{align*}
    H_\eta^p&\define\left\{A\in\mathcal{S}^p:\,\|A\|_{H_\eta^p}\define\sup_{0\leq s<t\leq T}\frac{\|A_{st}\|_{L^p}}{|s-t|^\eta}<\infty\right\}, \\
    \bar{H}_{\bar{\eta}}^p&\define\left\{A\in\mathcal{S}^p:\,\vertiii{A}_{\bar{H}_{\bar{\eta}}^p}\define\sup_{0\leq s<u<t\leq T}\frac{\|\expec{\delta A_{sut}|\mathcal{F}_s}\|_{L^p}}{|s-t|^{\bar{\eta}}}<\infty\right\},
\end{align*}
where we have set $\delta A_{s,u,t}\define A_{s,t}-A_{s,u}-A_{u,t}$. With this notation we have the following version of the stochastic sewing lemma:
\begin{proposition}[{Stochastic Sewing Lemma \cite[Theorem 2.1 and Propostion 2.7]{Le2020}}]\label{prop:stochastic_sewing}
    Let $p\geq 2$, $\eta>\frac12$, and $\bar{\eta}>1$. Suppose that $A\in H_\eta^p\cap\bar{H}_{\bar{\eta}}^p$. Then, for every $0\leq s\leq t\leq T$, the limit
    \begin{equation*}
        I_{s,t}(A)\define\lim_{|P|\to 0}\sum_{[u,v]\in P}A_{u,v}
    \end{equation*}
    along partitions $P$ of $[s,t]$ with mesh $|P|\define\max_{[u,v]\in P}|v-u|$ tending to zero exists in $L^p$. The limiting process $I(A)$ is additive in the sense that $I_{s,u}(A)+I_{u,t}(A)=I_{s,t}(A)$ for all $0\leq s\leq u\leq t\leq T$. Furthermore, there is a constant $C=C(p,\eta,\bar{\eta})$ such that
    \begin{equation*}
        \|I_{s,t}(A)\|_{L^p}\leq C\left(\vertiii{A}_{\bar{H}_{\bar{\eta}}^p}|t-s|^{\bar{\eta}}+\|A\|_{H_{\eta}^p}|t-s|^\eta\right)
    \end{equation*}
    for all $0\leq s\leq t\leq T$. Moreover, if $\|\Expec{A_{s,t}\,|\,\F_s}\|_{L^p}\lesssim|t-s|^{\bar{\eta}}$, then $I(A)\equiv 0$.
\end{proposition}
Recall our notation of the fast motion's flow from \eqref{eq:general_flow} and \eqref{eq:general_flow-fixed-x}, respectively. We are ultimately going to apply \cref{prop:stochastic_sewing} with the two-parameter process
\begin{equation}\label{eq:riemann_summands}
  A_{s,t}^\varepsilon\define\int_s^t \left(g\Big(X_s^\varepsilon,\bar{\Phi}_{s,r}^{X_s^\varepsilon}\big(\Phi_{0,s}^{X^\varepsilon}(Y_0)\big)\Big)-\bar{g}\big(X_s^\varepsilon\big)\right)\,dB_r,\quad 0\leq s<t,
\end{equation}
where, thanks to the conditional independence of the integrand and $\sigma(B_r-B_s,r\in[s,t])$ given $\F_s$, the integral is well defined in the (mixed) Wiener-Young sense as detailed in \cref{sec:integral_estimates} below. There we also show that the integral $I(A^\varepsilon)$ constructed in \cref{prop:stochastic_sewing} actually agrees with the Young integral in \eqref{eq:slow_feedback_sec}. It will be clear that our bounds on $A^\varepsilon$ also apply to the Riemann summands for the drift term in \eqref{eq:slow_feedback_sec}, whence we exclude it from our considerations for now.

\subsubsection{A Priori Integral Estimates}\label{sec:integral_estimates}

We will use the notion of \emph{mixed Wiener-Young integrals}: If $F:(s,t]\to\Lin[m]{d}$ is a (sufficiently regular) random function \emph{independent} of $(\tilde{B}^s_r)_{r\in[0,t-s]}$, we can make the definition
\begin{equation}\label{eq:wiener_young}
  \int_s^t F_r\,dB_r\define\int_s^t F_r\dot{\bar{B}}_{r-s}^s\,dr+\int_0^{t-s} F_{r+s}\,d\tilde{B}_r^s,
\end{equation}
where the integral with respect to $\tilde{B}^s$ is well defined in the Wiener sense (after all $\tilde{B}^s$ is a Gaussian process). The H\"older norm of negative exponent $-\kappa$, $\kappa\in[0,1]$, is defined by 
\begin{equation*}
  |F|_{-\kappa}\define\sup_{u,v\in(s,t]}\frac{1}{|v-u|^{1-\kappa}}\left|\int_u^v F_r\,dr\right|.
\end{equation*}
Note that for $\kappa=0$ we of course recover the usual sup-norm $|F|_\infty$.

In terms of this norm, one can then prove the following fundamental estimate on \eqref{eq:wiener_young}:
\begin{lemma}[{\cite[Lemma 3.4]{Hairer2020}}]\label{lem:wiener_integral_bound}
  Let $2\leq p<q$. Fix $\kappa\in\big[0,H-\frac12\big)$ and $0\leq s\leq t\leq T$. Suppose that $F:(s,t]\to\Lin[m]{d}$ is independent of $(\tilde{B}^s_r)_{r\in[0,t-s]}$ and $|F|_{-\kappa}\in L^q$. Then one has the bound
  \begin{equation*}
    \left\|\int_s^t F_r\,dB_r\right\|_{L^p}\lesssim \||F|_{-\kappa}\|_{L^q}|t-s|^{H-\kappa},
  \end{equation*}
  where the prefactor is independent of $F$ and $0\leq s\leq t\leq T$.
\end{lemma}
We also have the the following estimate, which is a simple consequence of \cite[Lemmas 3.10 \& 3.12]{Hairer2020}:
\begin{lemma}\label{lem:stochastic_sewing_fbm}
  Let $p\geq 2$ and $\alpha>1-H$. Let $X$ be an $(\F_t)_{t\in[0,T]}$-adapted stochastic process with $\alpha$-H\"older sample paths. Moreover assume that $X\in\B_{\alpha,p}$. Let $f:\R^d\to\R$ be a bounded Lipschitz continuous function. Then we have the following bound on the Young integral:
  \begin{equation*}
    \left\|\int_s^t f(X_r)\,dB_r\right\|_{\B_{H,p}}\lesssim\big(|f|_\infty+\Lip{f}\big)\big(1+\|X\|_{\B_{\alpha,p}}\big),
  \end{equation*}
  uniformly in $0\leq s<t\leq T$.
\end{lemma}

It is of course fundamental for our argument that the `integral' furnished by \cref{prop:stochastic_sewing} indeed coincides with the Young integral. This is ensured by the two lemmas below.

\begin{lemma}\label{lem:fast_process_moments}
  Let $X=(X_t)_{t\in[0,T]}$ be a continuous process with values in $\R^d$. Let $b:\R^d\times\R^n\to\R^n$ be of linear growth and satisfy
  \begin{equation*}
   	\Braket{b(z, x)-b(z,y),x-y}\leq D- \kappa|x-y|^2 \qquad\forall\, x,y\in\R^n, \forall \,z\in \R^d.
  \end{equation*} 
  Then, for any $p\geq 2$ and any random variable $Y\in L^p$, the following holds:
  \begin{equation*}
	\sup_{\epsilon\in (0,1]}\sup_{0\leq s\leq t\leq T}\big\|\Phi^X_{s,t}(Y)\big\|_{L^p}\lesssim 1+\|Y\|_{L^p}+ \sup_{0\leq t\leq T} \|X_t\|_{L_p}.
  \end{equation*}
\end{lemma}
\begin{proof}
  It is clear that we can assume $s=0$ without loss of generality. Let $Z^\varepsilon$ solve
  \begin{equation*}
    Z_t^\varepsilon=Y-\frac{1}{\varepsilon}\int_0^t Z^\varepsilon_s\,ds+\frac{1}{\varepsilon^{\hat{H}}}\sigma\hat{B}_t.
  \end{equation*}
By \eqref{eq:solution_difference}, we have
  \begin{equation*}
    \big|\Phi_{0,t}^X(Y)-Z_t^\varepsilon\big|^2\lesssim\int_0^{\frac{t}{\varepsilon}} e^{-\kappa\left(\frac{t}{\varepsilon}-s\right)}\big(1+|X_{\varepsilon s}|+|Z_{\varepsilon s}^\varepsilon|\big)^2\,ds
  \end{equation*}
  for all $t\in[0,T]$. Since $(Z_{\varepsilon h}^\varepsilon)_{h\geq 0}\overset{d}{=}(Z_h^1)_{h\geq 0}$ and $\sup_{t\geq 0}\|Z^1_{t}\|_{L^p}\lesssim 1+\|Y\|_{L^p}$, the lemma follows at once.
\end{proof}

\begin{lemma}\label{lem:sewing_young}
  Let $H>\frac12$ and let $h:\R^d\times\R^n \to\R$ be a Lipschitz continuous function. Let $p>2$ and $\alpha>1-H$. Let $X$ be an $\R^d$-valued, $(\F_t)_{t\in[0,T]}$-adapted process with $\sup_{t\in[0,T]}\|X_t\|_{L^p}<\infty$ and sample paths in $\C^\alpha([0,T],\R^d)$. Let $Y_0\in L^p$. Define
  \begin{equation*}
    A_{s,t}\define\int_s^t h\Big(X_s,\bar{\Phi}_{s,r}^{X_s}\big(\Phi_{0,s}^{X}(Y_0)\big)\Big)\,dB_r,
  \end{equation*}
  where the integration is understood in the mixed Wiener-Young sense, see \eqref{eq:wiener_young}. If $A\in H_\eta^2\cap\bar{H}_{\bar{\eta}}^2$ for some $\eta>\frac12$ and $\bar{\eta}>1$, then, for any $\varepsilon>0$ and any $0\leq s\leq t\le T$,  
  \begin{equation*}
  \lim_{|P|\to 0} \sum_{[u,v]\in P([s,t])} A_{u,v}=\int_{s}^t h\big(X_r,\Phi_{0,r}^{X}(Y_0)\big)\,dB_r,
  \end{equation*} 
  where the right-hand side is the Young integral.
\end{lemma}
\begin{proof}
  We first note that, by \cref{lem:comparison}, the process $\Phi_{0,\cdot}^X(Y_0)$ takes values in $\C^\beta([0,T],\R^d)$ for any $\beta<\hat{H}$. The pathwise Young integral $\int h\big(X_r,\Phi_{0,r}^{X}(Y_0)\big)\,dB_r$ is thus well defined and is given by the limit of the Riemann sums of
  \begin{equation*}
    \tilde{A}_{s,t}\define h\big(X_s,\Phi_{0,s}^X(Y_0)\big)(B_t-B_s)
  \end{equation*}
  along any sequence of partitions. By the last part of \cref{prop:stochastic_sewing}, it now suffices to show that $\|A_{s,t}-\tilde{A}_{s,t}\|_{L^2}\lesssim |t-s|^{\bar{\eta}}$ for some $\bar{\eta}>1$. 

  To see this, we apply \cref{lem:wiener_integral_bound} with $\kappa=0$ to find that, for each $\beta<\hat{H}$,
  \begin{align*}
    &\phantom{\leq}\big\|A_{s,t}-\tilde{A}_{s,t}\big\|_{L^2}=\left\|\int_s^t \Big(h\big(X_s,\bar{\Phi}_{s,r}^{X_s}\big(\Phi_{0,s}^{X}(Y_0)\big)\big)-h\big(X_s,\Phi_{0,s}^X(Y_0)\big)\Big)\,dB_r\right\|_{L^2}\\
    &\leq\Big\|\sup_{s\leq r\leq t}\Big|h\big(X_s,\bar{\Phi}_{s,r}^{X_s}\big(\Phi_{0,s}^{X}(Y_0)\big)\big)-h\big(X_s,\Phi_{0,s}^X(Y_0)\big)\Big|\Big\|_{L^p}|t-s|^H\\
    &\leq\Lip{h}\Big\|\Big|\bar{\Phi}_{s,\cdot}^{X_s}\big(\Phi_{0,s}^{X}(Y_0)\big)\Big|_{\C^{\beta}}\Big\|_{L^p}|t-s|^{H+\beta}.
  \end{align*}
  Since $H+\hat{H}>1$, we can conclude with \cref{lem:comparison,lem:fast_process_moments}.
\end{proof}
Our interest in \cref{lem:sewing_young} is of course in applying it to the slow motion \eqref{eq:slow_feedback_sec} and the Riemann summands $A^\varepsilon_{s,t}$ defined in \eqref{eq:riemann_summands}. We have already seen in \cref{cor:norm_bound_solution} that $X^\varepsilon\in\bigcap_{p\geq 1}\B_{\alpha,p}$ for any $\alpha<\frac12\wedge\hat{H}$. We are therefore left to check that $A^\varepsilon\in H_\eta^p\cap\bar{H}_{\bar{\eta}}^p$ for some $\eta>\frac12$, $\bar{\eta}>1$, and $p\geq 2$. Since these estimates are somewhat technically involved and require longer computations, we devote a subsection to each of the norms $\|\cdot\|_{H_{\eta}^p}$ and $\vertiii{\;\cdot\;}_{\bar{H}_{\bar{\eta}}^p}$, respectively.

\subsubsection{Controlling the Increment $A^\varepsilon_{s,t}$}\label{sec:sewing_1}

Let $h:\R^d\times \R^n\to \R^d$. Recall that write $\bar h(x)=\int h(x,y) \pi^x(dy)$ for its average with respect to the first marginal of the invariant measure of the process $\bar \Phi^x$, see \eqref{eq:general_flow-fixed-x} and \cref{initial-condition}.
The following lemma exploits the convergence rates derived in \cref{sec:convergence}. [The reader should observe that without further notice we assume that the conditions of \cref{thm:feedback_fractional} on the drift $b:\R^d\times\R^n\to\R^n$ are in place.]
\begin{lemma}\label{lem:sewing_helper_1}
  Let $q>1$. Let $h:\R^d\times\R^n\to\R$ be a bounded measurable function and let $X,Y\in L^q$ be $\F_s$-measurable random variables. Then, for any $0\leq s\leq t$, any $p\geq 2$, and any $\zeta<1-\hat{H}$, we have that
  \begin{equation*}
    \left\|\int_s^t \Big( h\big(X,\bar{\Phi}_{s,r}^X(Y)\big)-\bar{h}(X)\Big) \,dr\right\|_{L^p}\lesssim |h|_\infty\Big(1+\|Y\|_{L^q}^{\frac{1}{p}}+\|X\|_{L^q}^{\frac{1}{p}}\Big)\varepsilon^{\frac{\zeta}{3p}}|t-s|^{1-\frac{\zeta}{3p}}.
  \end{equation*}
\end{lemma}
\begin{proof}
There is no loss of generality in assuming that $\bar{h}\equiv 0$. Notice also that the trivial estimate $\big\|\int_s^t h\big(X,\bar{\Phi}_{s,r}^X(Y)\big)\,dr\big\|_{L^\infty}\leq|h|_\infty|t-s|$. By interpolation, we can therefore restrict ourselves to the case $p=2$. Clearly,
  \begin{equation*}
    \Expec{\left|\int_s^t h\big(X,\bar{\Phi}_{s,r}^X(Y)\big)\,dr\right|^2}=2\int_s^t\int_s^v\Expec{h\big(X,\bar{\Phi}_{s,r}^X(Y)\big)h\big(X,\bar{\Phi}_{s,v}^X(Y)\big)}\,dr\,dv.
  \end{equation*}
  For $r<v$ we condition the integrand on $\F_r$, and use \cref{cor:total_variation_conditional} \ref{it:ergodicity_tv} together with \cref{lem:smooth_part_decay,lem:fast_process_moments} to find
  \begin{align*}
    &\phantom{\lesssim}\Big|\Expec{h\big(X,\bar{\Phi}_{s,r}^X(Y)\big)h\big(X,\bar{\Phi}_{s,v}^X(Y)\big)}\Big|\leq |h|_\infty\,\Expec{\,\Big|\Expec{h\big(X,\bar{\Phi}_{s,v}^X(Y)\big)\,\middle|\,\F_r}\Big|\,}\\
    &\lesssim |h|_\infty^2\,\Expec{\Big(1+\|\varepsilon^{-\hat{H}}\bar{\hat{B}}_{\varepsilon\cdot}^r\|_{\Omega_\zeta}\Big)\Big(1+\big|\bar{\Phi}_{s,r}^X(Y)\big|\Big)}\left(\frac{\varepsilon}{v-r}\right)^{\frac{\zeta}{3}}\\
    &\lesssim |h|_\infty^2 \big(1+\|Y\|_{L^q}+\|X\|_{L^q}\big) \left(\frac{\varepsilon}{v-r}\right)^{\frac{\zeta}{3}}. \qedhere
  \end{align*}
\end{proof}

We can now establish the required estimate on the $H_\eta^p$-norm of $A_{s,t}$.
\begin{proposition}\label{prop:sewing_1}
   Let $h:\R^d\times\R^n\to\R$ be bounded measurable and let $X$ be an $(\F_t)_{t\in[0,T]}$-adapted, continuous process with $\sup_{t\in[0,T]}\|X_t\|_{L^q}<\infty$ for some $q\geq 1$. Define 
  \begin{equation*}
    A_{s,t}\define\int_s^t  \left[h\Big(X_s,\bar{\Phi}_{s,r}^{X_s}\big(\Phi_{0,s}^X(Y_0)\big)\Big)-\bar{h}(X_s)\right]\,dB_r,\quad 0\leq s\leq t\leq T,
  \end{equation*}
  in the mixed Wiener-Young sense, see \eqref{eq:wiener_young}. Let $\kappa\in(0,H-\frac12)$ and set $\eta=H-\kappa$. Then $A\in H_\eta^p$ for each $p\geq 2$, and any $\varepsilon>0$. Moreover, there is a $\gamma>0$ such that 
  \begin{equation*}
    \|A\|_{H_{\eta}^p}\lesssim |h|_\infty\Big(1+\sup_{0\leq t\leq T}\|X_t\|_{L^q}\Big)\varepsilon^{\gamma}.
  \end{equation*}
\end{proposition}
\begin{proof}
  Again, we may assume that $\bar{h}\equiv 0$ without any loss of generality. Since $X$ is $(\F_t)_{t\in[0,T]}$-adapted, we can use \cref{lem:wiener_integral_bound} to obtain that, for $\tilde{q}>p$ and $\kappa\in[0,H-\frac12)$,
  \begin{equation*}
    \|A_{s,t}\|_{L^p}\lesssim\left\|\left|h\Big(X_s,\bar{\Phi}_{s,\cdot}^{X_s}\big(\Phi_{0,s}^X(Y_0)\big)\Big)\right|_{-\kappa}\right\|_{L^{\tilde{q}}}|t-s|^{H-\kappa}.
  \end{equation*}
 By \cref{lem:sewing_helper_1,lem:fast_process_moments}, we obtain 
    \begin{equation*}
 \left\|\int_u^v h\Big(X_s,\bar{\Phi}_{s,r}^{X_s}\big(\Phi_{0,s}^X(Y_0)\big)\Big)\,dr\right\|_{L^{\tilde{q}}}\lesssim |h|_\infty\bigg(1+\|Y_0\|^{\frac{1}{\tilde{q}}}_{L^q}+\sup_{0\leq r\leq s}\|X_r\|_{L^q}^{\frac{1}{\tilde{q}}}\bigg)\varepsilon^{\frac{\zeta}{3\tilde{q}}}|v-u|^{1-\frac{\zeta}{3\tilde{q}}}
  \end{equation*}
  for all $u,v\in[s,t]$ and any $\zeta<1-\hat{H}$. Therefore, Kolmogorov's continuity theorem shows that
  \begin{equation*}
    \left\|\left|h\Big(X_s,\bar{\Phi}_{s,\cdot}^{X_s}\big(\Phi_{0,s}^X(Y_0)\big)\Big)\right|_{-\kappa}\right\|_{L^{\tilde{q}}}\lesssim |h|_\infty\bigg(1+\|Y_0\|^{\frac{1}{\tilde{q}}}_{L^q}+\sup_{0\leq t\leq T}\|X_t\|_{L^q}^{\frac{1}{\tilde{q}}}\bigg)\varepsilon^{\frac{\zeta}{3\tilde{q}}},
  \end{equation*}
  provided that we choose $\tilde{q}>\kappa^{-1}\left(1+\frac{\zeta}{3}\right)$, and the final result follows.
  \end{proof}

\subsubsection{Continuity of the Invariant Measures}
Let $\varepsilon>0$ and $s<t$. We write 
\begin{equation*}
   P([s,t];\varepsilon)\define\begin{cases}
   \left\{s+k\varepsilon:\,k=0,\dots,\left[\frac{t-s}{\varepsilon}\right]\right\}\cup\{t\},& t-s-\varepsilon\left[\frac{t-s}{\varepsilon}\right]\geq\frac{\varepsilon}{2},\\
   \left\{s+k\varepsilon:\,k=0,\dots,\left[\frac{t-s}{\varepsilon}\right]-1\right\}\cup\{t-\frac{\varepsilon}{2},t\},& t-s-\varepsilon\left[\frac{t-s}{\varepsilon}\right]<\frac{\varepsilon}{2}.
   \end{cases}
\end{equation*} 
Notice that the distance between two subsequent points $(t_i,t_{i+1})\in P([s,t];\varepsilon)$ satisfies $|t_{i+1}-t_i|\in[\frac{\varepsilon}{2},\varepsilon]$. Recall from \cref{cond:feedback} that the drift $b$ is assumed to be locally Lipschitz uniformly with respect to the second argument. We write
\begin{equation*}
  \Lip[K]{b}\define\sup_{\substack{|x_1|,|x_2|\leq K\\y\in\R^n}}|b(x_1,y)-b(x_2,y)|
\end{equation*}
for $K>0$. In order to keep the statements of the next lemmas concise, we shall freely absorb quantities independent of $0\leq s\leq t$ and $\varepsilon\in(0,1]$ into the prefactor hidden beneath $\lesssim$.

\begin{lemma}\label{lem:continuity}
  Let $p\geq 1$ and suppose that $b(x,\cdot)\in\S_p(\kappa,R)$ for all $x\in\R^d$. Let $X,\bar{X}\in L^\infty$, and $Y\in L^p$ be $\F_s$-measurable random variables. Then
  \begin{equation*}
    \left\|\bar{\Phi}_{s,t}^X(Y)-\bar{\Phi}^{\bar{X}}_{s,t}(Y)\right\|_{L^p}\lesssim\|X-\bar{X}\|_{L^{p}}.
  \end{equation*}
\end{lemma}
\begin{proof}
  We abbreviate $\Lambda\define\Lambda(\kappa,R,p)$ and observe that, for any $s\leq u\leq r$,
  \begin{align*}
    \frac{d}{dr}\Big|\bar{\Phi}_{u,r}^{X}(Y)-\bar{\Phi}_{u,r}^{\bar{X}}(Y)\Big|^2&=\frac{2}{\varepsilon}\Braket{b\big(X,\bar{\Phi}_{u,r}^{X}(Y)\big)-b\big(\bar{X},\bar{\Phi}_{u,r}^{\bar{X}}(Y)\big),\bar{\Phi}_{u,r}^{X}(Y)-\bar{\Phi}_{u,r}^{\bar{X}}(Y)}\\
    &\leq \frac{2(\Lambda+1)}{\varepsilon}\Big|\bar{\Phi}_{u,r}^{X}(Y)-\bar{\Phi}_{u,r}^{\bar{X}}(Y)\Big|^2+ \frac{\Lip[\|X\|_{L^\infty}\vee\|\bar{X}\|_{L^\infty}]{b}^2}{2\varepsilon}|X-\bar{X}|^2
  \end{align*}
  with probability $1$. It follows that
  \begin{equation}\label{eq:cont_interpolate_1}
    \Big|\bar{\Phi}_{u,r}^X(Y)-\bar{\Phi}^{\bar{X}}_{u,r}(Y)\Big|\lesssim\Lip[\|X\|_{L^\infty}\vee\|\bar{X}\|_{L^\infty}]{b}e^{(\Lambda+1)\frac{|r-u|}{\varepsilon}}|X-\bar{X}|.
  \end{equation}
  This bound is of course only useful on a time interval with length of order $\varepsilon$. We therefore expand
  \begin{equation*}
    \left\|\bar{\Phi}_{s,t}^X(Y)-\bar{\Phi}^{\bar{X}}_{s,t}(Y)\right\|_{L^p}\leq\sum_{(t_i,t_{i+1})\in P([s,t];\varepsilon)}\left\|\bar{\Phi}_{t_{i+1},t}^{\bar{X}}\big(\bar{\Phi}_{s,t_{i+1}}^X(Y)\big)-\bar{\Phi}_{t_{i},t}^{\bar{X}}\big(\bar{\Phi}_{s,t_{i}}^X(Y)\big)\right\|_{L^p}.
  \end{equation*}
  \Cref{cor:fast_different_initial} shows that
  \begin{align*}
    \left\|\bar{\Phi}_{t_{i+1},t}^{\bar{X}}\big(\bar{\Phi}_{s,t_{i+1}}^X(Y)\big)-\bar{\Phi}_{t_{i},t}^{\bar{X}}\big(\bar{\Phi}_{s,t_{i}}^X(Y)\big)\right\|_{L^p}&\lesssim\Big\|\bar{\Phi}_{s,t_{i+1}}^X(Y)-\bar{\Phi}_{t_{i},t_{i+1}}^{\bar{X}}\big(\bar{\Phi}_{s,t_{i}}^X(Y)\big)\Big\|_{L^p} e^{-c\frac{|t-t_{i+1}|}{\varepsilon}}\\
    & \lesssim \|X-\bar{X}\|_{L^p}e^{-c\frac{|t-t_{i+1}|}{\varepsilon}},
  \end{align*}
  where the last inequality uses \eqref{eq:cont_interpolate_1} together with $|t_{i+1}-t_i|\asymp\varepsilon$. Consequently,
  \begin{equation*}
    \left\|\bar{\Phi}_{s,t}^X(Y)-\bar{\Phi}^{\bar{X}}_{s,t}(Y)\right\|_{L^p}\lesssim\|X-\bar{X}\|_{L^p}\sum_{(t_i,t_{i+1})\in P([s,t];\varepsilon)}e^{-c\frac{|t-t_{i+1}|}{\varepsilon}}\lesssim\|X-\bar{X}\|_{L^p}
  \end{equation*}
  uniformly in $0\leq s\leq t$ and $\varepsilon\in(0,1]$. 
\end{proof}

\Cref{lem:continuity} implies the local Lipschitz continuity of the invariant measure $\pi^x$ in the parameter $x\in\R^d$:
\begin{proposition}\label{lem:wasserstein_holder}
  Let $p\geq 1$ and $K>0$. Suppose that $b(x,\cdot)\in\S_p(\kappa,R)$ for all $x\in\R^d$. Then
  \begin{equation*}
    \W^p(\pi^{x_1},\pi^{x_2})\lesssim |x_1-x_2|,
  \end{equation*}
  uniformly for $|x_1|,|x_2|\leq K$.
\end{proposition}
\begin{proof}
  Owing to \cref{thm:geometric},  it follows that
  \begin{equation*}
     \W^p(\pi^{x_1},\pi^{x_2})\leq\limsup_{\varepsilon\to 0}\big\|\bar{\Phi}^{x_1}_{0,1}(0)-\bar{\Phi}^{x_2}_{0,1}(0)\big\|_{L^p}
   \end{equation*} 
   and we conclude with \cref{lem:continuity}.
\end{proof}

The simple proof of the following corollary is left to the reader.
\begin{corollary}\label{cor:lipschitz_average}
Let $h:\R^d\times \R^n\to \R^d$ be Lipschitz continuous. Then $\bar{h}:\R^d\to\R^d$ is locally Lipschitz.
\end{corollary}
\subsubsection{Controlling the Second Order Increment $\delta A^\varepsilon_{s,u,t}$}\label{sec:sewing_2}

Uniform bounds on the second order increments are difficult to obtain even for the Markovian fast dynamic. The first technical estimate of this subsection is the following:
\begin{lemma}\label{lem:sewing_helper_2}
  Let $1\leq p<q$ and suppose that $b(x,\cdot)\in\S_p(\kappa,R)$ for all $x\in\R^d$. Let $h:\R^d\times\R^n\to\R$ be a Lipschitz continuous function with $\bar{h}\equiv 0$. Suppose that $X,\bar{X}\in L^\infty$ and $Y\in L^q$ are $\mathcal{F}_s$-measurable random variables. Then, for any $\rho\in(0,1)$, there is a $\gamma>0$ such that 
    \begin{equation*}
      \Big\|\Expec{h\big(X,\bar{\Phi}^{X}_{s,t}(Y)\big)-h\big(\bar{X},\bar{\Phi}^{\bar{X}}_{s,t}(Y)\big)\,\middle|\,\mathcal{F}_s}\Big\|_{L^p}\lesssim\Lip{h}\big(1+\|Y\|_{L^q}\big)\|X-\bar{X}\|_{L^{p}}^{\rho}\left(1\wedge\frac{\varepsilon^\gamma}{|t-s|^\gamma}\right).
    \end{equation*}
\end{lemma}
\begin{proof}
  By \cref{cor:total_variation_conditional} \ref{it:ergodicity_wasserstein} and H\"older's inequality, we certainly have
  \begin{equation}\label{eq:sewing_helper_2_interpolate}
    \Big\|\Expec{h\big(X,\bar{\Phi}^{X}_{s,t}(Y)\big)-h\big(\bar{X},\bar{\Phi}^{\bar{X}}_{s,t}(Y)\big)\,\middle|\,\mathcal{F}_s}\Big\|_{L^p}\lesssim\Lip{h}\big(1+\|Y\|_{L^q}\big)\left(1\wedge\frac{\varepsilon^{\zeta}}{|t-s|^{\zeta}}\right).
  \end{equation}
  On the other hand, by the continuity lemma (\cref{lem:continuity}),
  \begin{align*}
    &\phantom{\lesssim}\Big\|\Expec{h\big(X,\bar{\Phi}^{X}_{s,t}(Y)\big)-h\big(\bar{X},\bar{\Phi}^{\bar{X}}_{s,t}(Y)\big)\,\middle|\,\mathcal{F}_s}\Big\|_{L^p}\\
    &\lesssim\Lip{h}\left(\|X-\bar{X}\|_{L^p}+\Big\|\bar{\Phi}^{X}_{s,t}(Y)-\bar{\Phi}^{\bar{X}}_{s,t}(Y)\Big\|_{L^p}\right)\lesssim\Lip{h}\|X-\bar{X}\|_{L^{p}}.
  \end{align*}
  Finally, we interpolate this bound with \eqref{eq:sewing_helper_2_interpolate}.
\end{proof}

Our remaining task is to derive an estimate on the distance between $\Phi^Z_{s,t}$ and $\bar{\Phi}_{s,t}^{Z_s}$. This is based on the following version of \cref{lem:continuity}:
\begin{lemma}\label{lem:continuity_path}
  Let $p\geq 1$ and suppose that $b(x,\cdot)\in\S_p(\kappa,R)$ for all $x\in\R^d$. Let $Y\in L^p$ be $\F_s$-measurable and $Z$ be a continuous process. Assume that $|Z|_\infty\in L^\infty$. Then
  \begin{equation*}
    \Big\|\bar{\Phi}^{Z_s}_{s,t}(Y)-\Phi_{s,t}^Z(Y)\Big\|_{L^p}\lesssim\Big\|\sup_{r\in[s,t]}|Z_r-Z_s|\Big\|_{L^{p}}
  \end{equation*}
\end{lemma}
\begin{proof}
  The reader can easily check that the very same argument we gave at the beginning of the proof of \cref{lem:continuity} also shows that, for $0\leq s\leq u\leq r\leq T$,
  \begin{align*}
    \Big|\bar{\Phi}_{u,r}^{Z_s}(Y)-\Phi_{u,r}^Z(Y)\Big|&\lesssim\Lip[\||Z|_{\infty}\|_{L^\infty}]{b}\left(\int_{\frac{u}{\varepsilon}}^{\frac{r}{\varepsilon}} e^{2(\Lambda+1)\left(\frac{r}{\varepsilon}-v\right)}|Z_{\varepsilon v}-Z_s|^2\,dv\right)^{\frac12}\\
    &\lesssim \sup_{v\in[u,r]}|Z_v-Z_s| e^{(\Lambda+1)\frac{|r-u|}{\varepsilon}}.
  \end{align*}
  The asserted bound then follows along the same lines as \cref{lem:continuity}.
\end{proof}

The following estimate is now an easy consequence:
\begin{lemma}\label{lem:sewing_helper_3}
  Let $p\geq 1$ and suppose that $b(x,\cdot)\in\S_p(\kappa,R)$ for all $x\in\R^d$. Let $h:\R^d\times\R^n\to\R$ be Lipschitz continuous. Assume furthermore that $X$ and $Y$ are $\F_u$- and $\F_s$-measurable random variables, respectively. Moreover, let $Z\in\B_{\alpha, p}([0,T],\R^d)$ for some $\alpha>0$ and assume that $|Z|_\infty\in L^\infty$. Then
    \begin{equation}\label{eq:sewing_moderately_3}
      \left\|\Expec{h\Big(X,\bar{\Phi}^{X}_{u,t}\big(\bar{\Phi}_{s,u}^{Z_s}(Y)\big)\Big)-h\Big(X,\bar{\Phi}^{X}_{u,t}\big(\Phi_{s,u}^{Z}(Y)\big)\Big)\,\middle|\,\mathcal{F}_u}\right\|_{L^{p}}\lesssim \Lip{h}\|Z\|_{\B_{\alpha,p}}|u-s|^{\alpha}e^{-c\frac{|t-u|}{\varepsilon}}.
    \end{equation} 
\end{lemma}
\begin{proof}
  By \cref{cor:fast_different_initial}, we have that
  \begin{align*}
    &\phantom{\lesssim}\Big\|\Expec{h\Big(X,\bar{\Phi}^{X}_{u,t}\big(\bar{\Phi}_{s,u}^{Z_s}(Y)\big)\Big)-h\Big(X,\bar{\Phi}^{X}_{u,t}\big(\Phi_{s,u}^{Z}(Y)\big)\Big)\,\middle|\,\mathcal{F}_u}\Big\|_{L^p}\\
    &\lesssim \Lip{h}\Big\|\bar{\Phi}_{s,u}^{Z_s}(Y)-\Phi_{s,u}^{Z}(Y)\Big\|_{L^p}e^{-c\frac{|t-u|}{\varepsilon}}.
  \end{align*}
  By \cref{lem:continuity_path},
  \begin{equation*}
    \Big\|\bar{\Phi}_{s,u}^{Z_s}(Y)-\Phi_{s,u}^{Z}(Y)\Big\|_{L^p}\lesssim\|Z\|_{\B_{\alpha,p}}|u-s|^{\alpha}.\qedhere
  \end{equation*}
\end{proof}

Finally, we can establish the second estimate needed for the application of \cref{prop:stochastic_sewing}:
\begin{proposition}\label{prop:sewing_2}
  Let $1\leq p<q$ and suppose that $b(x,\cdot)\in\S_q(\kappa,R)$ for all $x\in\R^d$. Let $h:\R^d\times\R^n\to\R$ be a Lipschitz continuous function. Assume that $X\in\B_{\alpha,p}$ for some $\alpha>1-H$ and $|X|_\infty\in L^\infty$. Define 
  \begin{equation*}
    A_{s,t}\define\int_s^t \bigg(h\Big(X_s,\bar{\Phi}_{s,r}^{X_s}\big(\Phi_{0,s}^X(Y_0)\big)\Big)-\bar{h}(X_s)\bigg)\,dB_r,
  \end{equation*}
  in the mixed Wiener-Young sense, see \eqref{eq:wiener_young}. Then $A\in\bar{H}_{\bar{\eta}}^p$ for any $\bar{\eta}<\alpha+H$ and any $\varepsilon>0$. Moreover, there is a $\gamma>0$ such that
  \begin{equation*}
    \vertiii{A}_{\bar{H}_{\bar{\eta}}^p}\lesssim\Lip{h}\big(1\vee\||X|_\infty\|_{L^\infty}\big)\big(1\vee\|X\|_{\B_\alpha,p}\big)\varepsilon^{\gamma}.
  \end{equation*}
\end{proposition}
\begin{proof}
  Fix $1<\bar{\eta}<\alpha+H$ and choose $\rho\in(0,1)$ such that $\bar{\eta}<H+\alpha\rho^2$ and $p\leq\rho q$. Since $|X|_\infty\in L^\infty$, owing to \cref{cor:lipschitz_average} we may assume that $\bar{h}\equiv 0$ without any loss of generality. Recall that
  $\delta A_{s,u,t}=A_{s,t}-A_{s,u}-A_{u,t}$, so
  \begin{equation*}
    \delta A_{s,u,t}=\int_u^t \bigg(h\Big(X_s,\bar{\Phi}_{s,r}^{X_s}\big(\Phi_{0,s}^X(Y_0)\big)\Big)-h\Big(X_u,\bar{\Phi}_{u,r}^{X_u}\big(\Phi_{0,u}^X(Y_0)\big)\Big)\bigg)\,dB_r.
  \end{equation*}
  We condition on $\F_u$ instead of $\F_s$. This gives
  \begin{align*}
    &\phantom{\leq}\|\Expec{\delta A_{s,u,t}\,|\,\mathcal{F}_s}\|_{L^p}\\
    &\leq \left\|\int_u^t \Expec{h\Big(X_s,\bar{\Phi}_{s,r}^{X_s}\big(\Phi_{0,s}^X(Y_0)\big)\Big)-h\Big(X_u,\bar{\Phi}_{u,r}^{X_u}\big(\Phi_{0,u}^X(Y_0)\big)\Big)\,\middle|\,\mathcal{F}_u}\dot{\bar{B}}^u_{r}\,dr\right\|_{L^p}\\
    &\leq\rom{1}+\rom{2}
  \end{align*}
  with
  \begin{align*}
    \rom{1}&\define\left\|\int_u^t \Expec{h\Big(X_s,\bar{\Phi}^{X_s}_{s,r}\big(\Phi_{0,s}^X(Y_0)\big)\Big)-h\Big(X_u,\bar{\Phi}^{X_u}_{u,r}\big(\bar{\Phi}_{s,u}^{X_s}\big(\Phi_{0,s}^X(Y_0)\big)\big)\Big)\,\middle|\,\mathcal{F}_u} \dot{\bar{B}}^u_{r}\,dr\right\|_{L^p},\\
    \rom{2}&\define\left\|\int_u^t \Expec{h\Big(X_u,\bar{\Phi}^{X_u}_{u,r}\big(\bar{\Phi}_{s,u}^{X_s}\big(\Phi_{0,s}^X(Y_0)\big)\big)\Big)-h\Big(X_u,\bar{\Phi}_{u,r}^{X_u}\big(\Phi_{0,u}^X(Y_0)\big)\Big)\,\middle|\,\mathcal{F}_u} \dot{\bar{B}}^u_{r}\,dr\right\|_{L^p}.
  \end{align*}
  These terms are now bounded individually. Let us begin with the bound on $\rom{1}$. Thanks to \cref{lem:fast_process_moments}, this term falls in the regime of \cref{lem:sewing_helper_2}. By H\"older's inequality, we therefore find
  \begin{align*}
    \rom{1}&\lesssim\Lip{h}\|X_s-X_u\|_{L^\infty}^{\rho(1-\rho)}\|X_s-X_u\|_{L^p}^{\rho^2}\int_u^t\big\|\dot{\bar{B}}_r^u\big\|_{L^{\frac{p}{1-\rho}}}\left(1\wedge\frac{\varepsilon^\gamma}{|r-u|^\gamma}\right)\,dr\\
    &\lesssim\Lip{h}\||X|_\infty\|_{L^\infty}^{\rho(1-\rho)}\|X\|_{\B_{\alpha,p}}^{\rho^2}\varepsilon^\delta|t-s|^{\bar{\eta}}
  \end{align*}
  for $\delta>0$ sufficiently small. Here, the last inequality used that, for any $p\geq 1$, $\big\|\dot{\bar{B}}_r^u\big\|_{L^{p}}\lesssim |r-u|^{H-1}$ together with the elementary fact
  \begin{equation*}
    \int_u^t\frac{1}{|r-u|^{1-H}}\left(1\wedge\frac{\varepsilon^\gamma}{|r-u|^\gamma}\right)\,dr\lesssim\varepsilon^\delta|t-u|^{H-\delta}
  \end{equation*}
  for any $\delta\in(0,\gamma]$.

  The term $\rom{2}$ can be handled similarly in view of \cref{lem:sewing_helper_3}.
\end{proof}

\subsection{Proof of \cref{thm:feedback_fractional}}\label{sec:proof}

The estimates of the previous two subsection furnish the following fundamental estimates:
\begin{proposition}\label{prop:final_control}
  Let $2\leq p<q$ and suppose that $b(x,\cdot)\in\S_q(\kappa,R)$ for all $x\in\R^d$. Let $h:\R^d\times\R^n\to\R$ be a bounded Lipschitz continuous function. Assume that there is an $\alpha>1-H$ such that $X$ has $\alpha$-H\"older sample paths and $X\in\B_{\alpha,p}$. If, in addition, $|X|_\infty\in L^\infty$, then, for any $\eta<H$ and any $\bar{\eta}<\alpha+H$, there is a $\gamma>0$ such that
  \begin{equation}
    \left\|\int_0^\cdot \Big(h\big(X_r,\Phi_{0,s}^X(Y_0)\big)-\bar{h}(X_r)\Big)\,dB_r\right\|_{\B_{\eta,p}}\lesssim\big(|h|_\infty+\Lip{h}\big)\big(1+\||X|_\infty\|_{L^\infty}\big)\big(1+\|X\|_{\B_{\alpha,p}}\big)\varepsilon^{\gamma},\label{eq:combine_sewing_1}
  \end{equation}
  and
  \begin{equation}\label{eq:combine_sewing_2}
    \left\|\int_0^\cdot h\big(X_r,\Phi_{0,r}^X(Y_0)\big)\,dB_r\right\|_{\B_{\eta,p}}\lesssim\big(|h|_\infty+\Lip{h}\big)\big(1+\||X|_\infty\|_{L^\infty}\big)\big(1+\|X\|_{\B_{\alpha,p}}\big),
  \end{equation}
  uniformly in $0\leq s<t\leq T$ and $\varepsilon\in(0,1]$. Here, the integrals are both taken in the Young sense.
\end{proposition}
\begin{proof}
  First note that, by \cref{lem:sewing_young}, the Young integrals in both \eqref{eq:combine_sewing_1} and \eqref{eq:combine_sewing_2} coincide with the processes $I(A^i)$ obtained by `sewing' the Riemann summands
  \begin{equation*}
  	A_{s,t}^1\define\int_s^t \Big(h\Big(X_s,\bar{\Phi}_{s,r}^{X_s}\big(\Phi_{0,s}^X(Y_0)\big)\Big)-\bar{h}(X_s)\Big)\,dB_r,\qquad A^2_{s,t}\define\int_s^t h\Big(X_s,\bar{\Phi}_{s,r}^{X_s}\big(\Phi_{0,s}^X(Y_0)\big)\Big)\,dB_r,
  \end{equation*}
  where integration is now understood in the mixed Wiener-Young sense (see \eqref{eq:wiener_young}), with the help of \cref{prop:stochastic_sewing}. Consequently, the estimate \eqref{eq:combine_sewing_1} follows immediately from combining \cref{prop:sewing_1,prop:sewing_2}. Owing to \cref{lem:stochastic_sewing_fbm,cor:lipschitz_average}, \eqref{eq:combine_sewing_2} is then an easy consequence of the first bound. 
\end{proof}

For $\varepsilon>0$ and $M>0$, let us define the $(\F_t)_{t\geq 0}$-stopping time $\tau_M^\varepsilon\define\inf\{t\geq 0:\,|X_t^\varepsilon|>M\}$. Applying the previous proposition to the slow-fast system \eqref{eq:slow_feedback_sec}--\eqref{eq:fast_feedback_sec}, we can deduce relative compactness of the stopped slow motion $X^{\varepsilon,M}\define X^\varepsilon_{\cdot\wedge\tau_M^\varepsilon}$:

\begin{corollary}\label{cor:tightness}
  Consider the slow-fast system \eqref{eq:slow_feedback_sec}--\eqref{eq:fast_feedback_sec} with \cref{cond:feedback} in place. Let $\beta<\frac12\wedge\hat{H}$ and $p\geq 2$. Suppose that there are $\kappa,R>0$ and $q>p$ such that $b(x,\cdot)\in\S_q(\kappa,R)$ for each $x\in\R^d$. Then, for any $M>0$,
  \begin{equation*}
    \sup_{\varepsilon\in(0,1]}\big\|X^{\varepsilon,M}\big\|_{\B_{\beta,p}}<\infty.
  \end{equation*}
\end{corollary} 
\begin{proof}
  Recall from \cref{cor:norm_bound_solution} that, for each $\varepsilon>0$, there is a unique global solution $X^\varepsilon$ to \eqref{eq:slow_feedback_sec} with values in $\C^{\alpha}([0,T],\R^d)$ for some $\alpha>1-H$. Moreover, since the H\"older norm of the stopped solution $X^{\varepsilon,M}$ is controlled by the H\"older norm of $X^\varepsilon$, the argument of \cref{cor:norm_bound_solution} also shows that $\big\|X^{\varepsilon,M}\big\|_{\B_{\beta,p}}<\infty$ for each $\beta<\frac12\wedge\hat{H}$ and $p\geq 1$. Employing \cref{prop:final_control}, we obtain that, for any $\gamma<H-\beta$ and any $\delta\in(0,T]$,
  \begin{equation*}
    \big\|X^{\varepsilon,M}\restriction_{[0,\delta]}\big\|_{\B_{\beta,p}}\lesssim \big(|g|_\infty+\Lip{g}\big)\big(1+\big\|X^{\varepsilon,M}\restriction_{[0,\delta]}\big\|_{\B_{\beta,p}}\big)\delta^\gamma+|f|_\infty\delta^{1-\beta},
  \end{equation*}
  uniformly in $\varepsilon\in(0,1]$. Hence, choosing $\delta>0$ sufficiently small, the proof is concluded by a standard iteration argument.
\end{proof}

Now we can finish the proof of \cref{thm:feedback_fractional} by localizing the argument of Hairer and Li. To this end, we rely on the following deterministic residue lemma:
\begin{lemma}[Residue Lemma]\label{lem:residue_lemma}
  Let $F:\R^d\to\R^d$ be Lipschitz continuous, $G:\R^d\to\Lin[m]{d}$ be of class $\C_b^2$, and $\h\in\C^\alpha([0,T],\R^n)$ for some $\alpha>\frac12$. Moreover, let $Z,\bar{Z}\in\C^{\tilde{\alpha}}([0,T],\R^d)$ for some $\tilde{\alpha}\in(1-\alpha,\alpha]$ with $Z_0=\bar{Z}_0$. Then there is a constant $C$ depending only on $F$, $G$, and the terminal time $T$ such that
  \begin{equation*}
    |z-\bar{z}|_{\C^{\tilde{\alpha}}}\leq C\exp\left(C|\h|_{\C^\alpha}^{\frac1\alpha}+C|Z|_{\C^{\tilde{\alpha}}}^{\frac{1}{\tilde{\alpha}}}+C|\bar{Z}|_{\C^{\tilde{\alpha}}}^{\frac{1}{\tilde{\alpha}}}\right)|Z-\bar{Z}|_{\C^{\tilde{\alpha}}},
  \end{equation*}
  where $z$ and $\bar{z}$ are the solutions to the equations
  \begin{equation*}
    z_t=Z_t+\int_0^t F(z_s)\,ds+\int_0^t G(z_s)\,d\h_s,\qquad\bar{z}_t=\bar{Z}_t+\int_0^t F(z_s)\,ds+\int_0^t F(\bar{z}_s)\,d\h_s.
  \end{equation*}
\end{lemma}
Albeit the statement of \cref{lem:residue_lemma} is slightly stronger than \cite[Lemma 2.2]{Hairer2020}, it is straight-forward to show that the very same proof still applies. We therefore omit the details and finally turn to the proof of the main result of this article:

\begin{proof}[{Proof of \cref{thm:feedback_fractional}}]
  First observe that, by the assumptions of the theorem and \cref{cor:lipschitz_average}, there exists a unique global solution to the averaged equation \eqref{eq:effective_dynamics}, see \cite{Lyons1998,Lyons2002,Rascanu2002}. We fix $\bar{\alpha}\in(\alpha,H)$ with $(\bar\alpha-\alpha)^{-1}<p$. Choose $\beta\in(1-H,\hat{H}\wedge\frac12)$. By \cref{cor:tightness}, $\sup_{\varepsilon\in(0,1]}\|X^{\varepsilon,M}\|_{\B_{\beta,p}}<\infty$ for each $M>0$. Consequently, by \cref{prop:final_control}, we deduce that
  \begin{align*}
    \left\|\int_0^\cdot \Big(g\big(X_r^{\varepsilon,M},\Phi_{0,r}^{X^{\varepsilon,M}}(Y_0)\big)-\bar{g}\big(X_r^{\varepsilon,M}\big)\Big)\,dB_r\right\|_{\B_{\bar{\alpha},p}}&\lesssim\varepsilon^\gamma,
    \\
    \left\|\int_0^\cdot \Big(f\big(X_r^{\varepsilon,M},\Phi_{0,r}^{X^{\varepsilon,M}}(Y_0)\big)-\bar{f}\big(X_r^{\varepsilon,M}\big)\Big)\,dr\right\|_{\B_{\bar{\alpha},p}}&\lesssim\varepsilon^\gamma.
  \end{align*}
  Therefore, $\big\|\hat{X}^{\varepsilon,M}-\bar{X}^{\varepsilon,M}\big\|_{\B_{\bar{\alpha},p}}\lesssim\varepsilon^\gamma$, where
  \begin{align*}
    \hat{X}^{\varepsilon,M}_t&\define X_0+\int_0^t f\big(X^{\varepsilon,M}_r,\Phi_{0,r}^{X^{\varepsilon,M}}(Y_0)\big)\,dr+\int_0^t g\big(X^{\varepsilon,M}_r,\Phi_{0,r}^{X^{\varepsilon,M}}(Y_0)\big)\,dB_r,\\
    \bar{X}^{\varepsilon,M}_t&\define X_0+\int_0^t\bar{f}\big(X^{\varepsilon,M}_r\big)\,dr+\int_0^t\bar{g}\big(X^{\varepsilon,M}_r\big)\,dB_r.
  \end{align*}
  In particular, $\big|\hat{X}^{\varepsilon,M}-\bar{X}^{\varepsilon,M}\big|_{\C^\alpha}\to 0$ in probability by the embedding \eqref{eq:embeddings}. Note also the decomposition
  \begin{equation*}
    X_t^{\varepsilon,M}=\hat{X}_t^{\varepsilon,M}-\bar{X}_t^{\varepsilon,M}+X_0+\int_0^t\bar{f}(X_r^{\varepsilon})\,dr+\int_0^t \bar{g}(X_r^{\varepsilon})\,dB_r,\quad t\in[0,\tau_M^\varepsilon\wedge T],
  \end{equation*}
  whence \cref{lem:residue_lemma} furnishes the bound
  \begin{equation}\label{eq:residue_bound}
    \big|X^{\varepsilon}-\bar{X}\big|_{\C^\alpha([0,\tau_M^\varepsilon\wedge T])}\leq C\exp\left(C|B|_{\C^\alpha}^{\frac{1}{\alpha}}+C\big|\hat{X}^{\varepsilon,M}-\bar{X}^{\varepsilon,M}\big|_{\C^\alpha}^{\frac{1}{\alpha}}\right)\big|\hat{X}^{\varepsilon,M}-\bar{X}^{\varepsilon,M}\big|_{\C^\alpha}.
  \end{equation}
  As we have seen above, for each $M>0$, the right-hand side goes to $0$ in probability as $\varepsilon\to 0$. Hence, we also have that $\big|X^{\varepsilon}-\bar{X}\big|_{\C^\alpha([0,\tau_M^\varepsilon\wedge T])}\to 0$ in probability.

  On the other hand, note that
  \begin{align}
  	\prob(\tau_M^\varepsilon<T)&\leq\prob\left(\sup_{t\in[0,\tau_M^\varepsilon]}\big|X_t^{\varepsilon}\big|\geq M,\tau_M^\varepsilon<T\right)\leq\prob\left(\big|X^{\varepsilon}\big|_{\C^\gamma([0,\tau_M^\varepsilon\wedge T])}\geq T^{-\gamma}(M-\|X_0\|_{L^\infty})\right)\nonumber\\
  	&\leq\prob\left(\big|X^{\varepsilon}-\bar{X}\big|_{\C^\gamma([0,\tau_M^\varepsilon\wedge T])}\geq T^{-\gamma}(M-\|X_0\|_{L^\infty})-|\bar{X}|_{\C^\gamma([0,T])}\right)\nonumber\\
  	&\leq\prob\left(\big|X^{\varepsilon}-\bar{X}\big|_{\C^\gamma([0,\tau_M^\varepsilon\wedge T])}\geq 1\right)+\prob\left(\big|\bar{X}\big|_{\C^\gamma([0,T])}>T^{-\gamma}(M-\|X_0\|_{L^\infty})-1\right)\label{eq:split}
  \end{align}
  for each $\gamma>0$. By \cref{prop:nualart}, we know that $\big|\bar{X}\big|_{\C^\gamma([0,T])}\in L^1$ provided that $\gamma<\frac12$. We fix such a $\gamma$.

  It is now easy to finish the proof. Let $\delta_1,\delta_2\in(0,1)$ be given. Then we can find a $M>0$ such that
  \begin{equation*}
  	\prob\left(\big|\bar{X}\big|_{\C^\gamma([0,T])}>T^{-\gamma}(M-\|X_0\|_{L^\infty})-1\right)\leq\frac{\delta_2}{2}.
  \end{equation*}
  For this $M$, we can also find an $\varepsilon_0>0$ such that
  \begin{equation*}
  	\prob\left(\big|X^{\varepsilon}-\bar{X}\big|_{\C^\alpha([0,\tau_M^\varepsilon\wedge T])}>\delta_1\right)\leq\frac{\delta_2}{4}\qquad\forall\,\varepsilon\in(0,\varepsilon_0).
  \end{equation*}
  The estimate \eqref{eq:split} therefore yields that
  \begin{align*}
  	\prob\left(\big|X^{\varepsilon}-\bar{X}\big|_{\C^\alpha([0,T])}>\delta_1\right)&\leq\prob\left(\big|X^{\varepsilon}-\bar{X}\big|_{\C^\alpha([0,\tau_M^\varepsilon\wedge T])}>\delta_1,\tau_M^\varepsilon\geq T\right)+\prob(\tau_M^\varepsilon<T)\\
  	&\leq 2\prob\left(\big|X^{\varepsilon}-\bar{X}\big|_{\C^\alpha([0,\tau_M^\varepsilon\wedge T])}>\delta_1\right)+\frac{\delta_2}{2}\leq\delta_2
  \end{align*}
  for all $\varepsilon\in(0,\varepsilon_0)$. Hence, $\big|X^{\varepsilon}-\bar{X}\big|_{\C^\alpha([0,T])}\to 0$ in probability as $\varepsilon\to 0$, as required.
\end{proof}

\begin{remark}
  The proof above shows that we can choose
  \begin{equation*}
    \lambda_0=\inf_{x\in\R^d}\Lambda(\kappa,R,p)
  \end{equation*}
  for any $p>\max\big(2,(H-\alpha)^{-1}\big)$ in \cref{thm:feedback_fractional}. Here, $\Lambda$ is the constant from \cref{prop:conditional_initial_condition_wasserstein}.
\end{remark}

\subsection{Smoothness of the Averaged Coefficients}

Let us finally show that an \emph{everywhere contractive} fast process falls in the regime of \cref{thm:feedback_fractional}. While smoothness of $\bar g$ also holds under less restrictive conditions, the proof becomes much more involved. To keep this article concise, we chose to report on these results in future work.

\begin{corollary}\label{cor:smooth}
  Suppose that
  \begin{itemize}
    \item $g\in\C_b^3\big(\R^d\times\R^n,\Lin[m]{d}\big)$,
    \item there is a $\kappa>0$ such that $b(x,\cdot)\in\S(\kappa,0,0)$ for every $x\in\R^d$,
    \item $b\in\C^3\big(\R^d\times\R^n,\R^d\big)$ is globally Lipschitz continuous and there is an $N\in\N$ such that, for each $i,j,k\in\{x,y\}$,
    \begin{equation*}
      |D^2_{i,j} b(x,y)|+|D^3_{i,j,k}b(x,y)|\lesssim 1+|y|^N\qquad\forall\,x\in\R^d,\,\forall\, y\in\R^n.
    \end{equation*}
  \end{itemize}
  Then the conclusion of \cref{thm:feedback_fractional} holds.
\end{corollary}
\begin{example}
  Let $V\in\C^4(\R^d\times\R^n)$. If $\inf_{x,y}D_{y,y}^2 V(x,y)\geq\kappa$, $|D^2_{x,y}V|_\infty+|D^2_{y,y}V|_\infty<\infty$, and, for each $i,j,k\in\{x,y\}$,
  \begin{equation*}
    |D^3_{i,j,y}V(x,y)|+|D^4_{i,j,k,y}V(x,y)|\lesssim 1+|y|^N\qquad\forall\,x\in\R^d,\,\forall\, y\in\R^n,
  \end{equation*}
  then $b=-D_y V$ falls in the regime of \cref{cor:smooth}. To give a concrete example, we can choose $V(x,y)=\big(2+\sin(x)\big)\big(y^2+\sin(y)\big)$, which furnishes the drift $b(x,y)=-\big(2+\sin(x)\big)\big(2y+\cos(y)\big)$.
\end{example}

\begin{proof}[Proof of \cref{cor:smooth}]
  In order to apply \cref{thm:feedback_fractional} it is enough to show that, for any $g\in\C_b^3(\R^n)$, the function
  \begin{equation*}
    \bar{h}(x)\define\int_{\R^n}g(y)\,\pi^x(dy)
  \end{equation*}
  is again of class $\C_b^2(\R^d)$. To this end, we define $h_t(x)\define\Expec{g(Y_t^x)}$ where $Y^x$ is the solution to the SDE
  \begin{equation*}
    dY_t^x=b(x,Y_t^x)\,dt+\sigma\,d\hat{B}
  \end{equation*}
  started in the generalized initial condition $\delta_0\otimes\sW$. Note that $h_t\to\bar{h}$ pointwise as $t\to\infty$ by \cref{thm:geometric}. Since $h_t\in\C_b^2(\R^d)$ for each $t\geq 0$, it thus suffices to show that 
  \begin{equation}\label{eq:derivative_bound}
    \sup_{t\geq 0} \left(|D h_t|_\infty+|D^2 h_t|_\infty\right)<\infty
  \end{equation} 
  and both $D h_t$ and $D^2 h_t$ converge locally uniformly along a subsequence. By a straight-forward `diagonal sequence' argument, we actually only need to prove uniform convergence on a fixed compact $K\subset\R^d$.

  Under the assumptions of the corollary, it is easy to see that the mapping $x\mapsto Y_t^x$ is three-times differentiable for each $t\geq 0$ and it holds that
  \begin{align*}
    D_x Y_t^x&=\int_0^tJ_{s,t}D_x b(x,Y_s^x)\,ds, \label{eq:first_derivative}\\
    D^2_{x,x} Y_t^x(u\otimes v)&=\int_0^tJ_{s,t}\Big(D_{x,x}^2 b(x,Y_s^x)(u\otimes v)+2D_{x,y}^2 b(x,Y_s^x)\big(u\otimes D_x Y_s^x(v)\big) \nonumber\\
    &\phantom{=\int_0^tJ_{s,t}}+D^2_{y,y}b(x,Y_s^x)\big(D_x Y_s^x(u)\otimes D_x Y_s^x(v)\big)\Big)\,ds,
  \end{align*}
  where $J_{s,t}$ solves the homogeneous problem
  \begin{equation*}
    J_{s,t}=\mathrm{id}+\int_s^t D_yb(x,Y_r^x)J_{s,r}\,dr.
  \end{equation*}
  Since $b(x,\cdot)\in\S(\kappa,0,0)$, it is not hard to see that, for each $x\in\R^d$ and $y\in\R^n$, $D_yb(x,y)\leq-\kappa$ in the sense of quadratic forms. In particular, the operator norm of $J$ satisfies the bound
  \begin{equation*}
    |J_{s,t}|\leq e^{-\kappa(t-s)}.
  \end{equation*}
  By an argument similar to \cref{lem:fast_process_moments}, it follows that, for any $p\geq 1$,
  \begin{equation*}
    \sup_{t\geq 0}\sup_{x\in\R^d}\big\|D_xY_t^x\big\|_{L^p}<\infty\quad\text{and}\quad \sup_{t\geq 0}\sup_{x\in\R^d}\big\|D_{x,x}^2Y_t^x\big\|_{L^p}<\infty.
  \end{equation*}
  Based on this, it is straight-forward to verify \eqref{eq:derivative_bound}. Consequently, by the Arzela-Ascoli theorem, there is a subsequence of times along which $Dh$ converges uniformly on $K$. By a similar---albeit more tedious---computation, the reader can easily check that also 
  \begin{equation*}
    \sup_{t\geq 0}\sup_{x\in\R^d}\big\|D_{x,x,x}^3Y_t^x\big\|_{L^p}<\infty.
  \end{equation*}
  In particular, $D^3 h$ is uniformly bounded, whence we can pass to a further subsequence along which $D^2 h$ also converges uniformly on $K$. Therefore, $\bar{h}\in\C_b^2(\R^d)$ as required.
\end{proof}

{
\footnotesize
\bibliographystyle{alpha}
\bibliography{./fenv-ave}
}

\end{document}